\documentclass[reqno]{amsart}

\usepackage{amssymb,latexsym}
\usepackage{amsfonts}
\usepackage{dsfont}
\usepackage[pdftex]{hyperref}
\usepackage{color}
\usepackage{graphicx}
\usepackage{parskip}
\usepackage{mathrsfs}
\usepackage{array}

\usepackage{tikz,graphicx, pgfplots}

\theoremstyle{plain}

\newtheorem{theorem}{Theorem}[section]
\newtheorem{lemma}{Lemma}[section]

\numberwithin{equation}{section}

\newcommand{\p}{\partial}

\newcommand{\rr}{\mathbb{R}}

\newcommand{\zz}{\mathbb{Z}}

\newcommand{\ci}{\mathbb{T}}

\newcommand{\wh}{\widehat}
\newcommand{\ee}{\varepsilon}
\newcommand{\vp}{\varphi}
\newcommand{\s}{\sigma}

\title{The Cauchy problem for the integrable RZQ equation}
\author{John Holmes, Katie Massey, and Ryan C. Thompson}
\date{May 2025}

\begin{document}
\begin{abstract}
    In this paper we study a new integrable fifth-order Camassa–Holm (CH)–type equation derived by Reyes, Zhu, and Qiao \cite{rzq}, which we call the RZQ equation. The m-form of this equation possesses a striking similarity to the m-form of the CH equation. However, unlike the CH equation, the nonlocal form of this equation cannot be interpreted as a nonlocal perturbation of Burgers' equation. 

    We prove that the initial value problem corresponding to the RZQ equation is well-posed in the sense of Hadamard, in Sobolev spaces $H^s$, $s>7/2$. We further show that the data-to-solution map is not uniformly continuous in the $H^s$ topology, though it is H\"older continuous in a weaker topology.  The initial value problem corresponding to the RZQ equation is ill-posed in $H^s$ for $s<7/2$.
\end{abstract} 
\maketitle

\section{Introduction}
  In this paper we study the following initial value problem (ivp) for the fifth order integrable  Camassa-Holm (CH) type equation, derived recently by Reyes, Zhu, and Qiao \cite{rzq} as in   \cite{qr}, which we shall hereafter refer to as the RZQ equation
\begin{align}\label{qrequation}
&m_t+vm_x+2v_xm=0, \quad t\in \mathbb R\\
&u(x,0)=u_0(x), \quad x\in \mathbb T \text{ or } \mathbb R, 
\end{align}  
where $m = \lambda ^2v $, $v = \lambda ^2u $, and $\lambda ^2= (1-\partial_x^2) $.  
By letting 
$\displaystyle \lambda^{-k} u  = \mathcal{F}^{-1}\left( \frac{\mathcal{F}(u)}{(1+\xi^2)^{k/2}}\right)$, $k \in \mathbb R$, where $\mathcal{F}$ is the Fourier transform operator and $\mathcal{F}^{-1}$ its inverse, we can write the ivp for the RZQ equation in the following nonlocal form
\begin{equation}\label{QR-nl}  
\begin{cases}
  u_t  + \lambda^{-2} [ \lambda^2 u \lambda ^2u_x ] 
  = 
  - \lambda^{-4} [   \lambda^2 u_x \lambda ^2u_{xx}  + 2 
 \lambda ^2  u \lambda^2 u_x   ]   
 \\
 u(x,0) = u_0(x).
 \end{cases}
\end{equation}

We seek to show that the RZQ equation is well-posed in the sense of Hadamard in Sobolev spaces $H^s$ for $s>7/2$.  Furthermore, we show that the continuity of the data-to-solution map is sharp.  Indeed, we find that while the mapping is continuous, it is not uniformly continuous.  Finally, we show that the RZQ equation is ill-posed in $H^s$ for $s<7/2$ on the real line.

Note that if we replace $v$ with $u$ in equation \eqref{qrequation}, we obtain the celebrated CH equation which first appeared in \cite{fuchssteiner1981} and models unidirectional propagation of shallow water waves over a flat bottom and arises as a bi-Hamiltonian integrable system with connections to geodesic flow on diffeomorphism groups \cite{camassa1993, const4, misiolek1998}.  We know the CH equation to be completely integrable with many useful conservation laws \cite{len1, mc1}.  Furthermore, local well-posedness has been shown in Sobolev spaces $H^s$ for $s>3/2$ and blow-up scenarios have been found \cite{const2, const3, jnz1, lo1, mc1}.  Global existence of solutions was also proved in \cite{bressan1, bressan2, jzz1}.  It is also worth noting that, like the RZQ equation, the CH equation has a corresponding nonlocal form
\begin{equation}\label{CH-nonloc}
    u_t+uu_x = -\lambda ^{-2} \p_x\left[u^2+\frac12u_x^2\right].
\end{equation}
It is from this form that we may see one of the most striking features of the CH equation which are solitary wave, or ``peakon," solutions of the form 
\begin{align}\label{ch-peakon}
    u(x,t) = ce^{-|x-ct|}.
\end{align}    These are weak solutions with a peaked profile and numerically equivalent wave speed and amplitude \cite{camassa1993}.  The orbital stability of these solutions was proven in \cite{const5}.

Recently, there have been explorations into higher order generalizations of the Camassa-Holm equation \cite{coc1, ding1, ding2, han1, holm1, mcl1, tang1, tian1, wang1, wang2}.  One that is of interest is the fifth order Camassa-Holm (FOCH) model \cite{lq1}
\begin{equation}\label{foch}
    \begin{cases}
        m_t+um_x+bu_xm = 0, \ \ \ \ \ \ \ \ \ \ \ \ t>0, \ x\in \rr, \\
        m=(1-\alpha^2\p_x^2)(1-\beta^2\p_x^2)u, \ \ \ t>0, \ x \in \rr,
    \end{cases}
\end{equation}
where $b\in\rr$ is a constant, $\alpha,\beta \in \rr$ are two parameters, $\alpha \neq 0$, $\alpha \beta \neq 0$.  Note that by letting $b=2$, $\alpha=1$, and $\beta=0$ we obtain the CH equation.  This model was developed by Liu and Qiao \cite{lq1} where they found explicit single pseudo-peakon, two-peakon, and $N$-peakon solutions along with their dynamical interactions.  The \textit{pseudo-peakons} are bounded weak traveling wave solutions with continuous first and second derivatives but whose higher order derivatives blow up \cite{qr}. In particular, the pseudo-peakon solution takes the form 
\begin{align}\label{qr-peakon}
     u(x,t) = ce^{-|x-ct|}(1+|x-ct|) .
\end{align}
\begin{figure}[h]
\includegraphics[width=6cm]{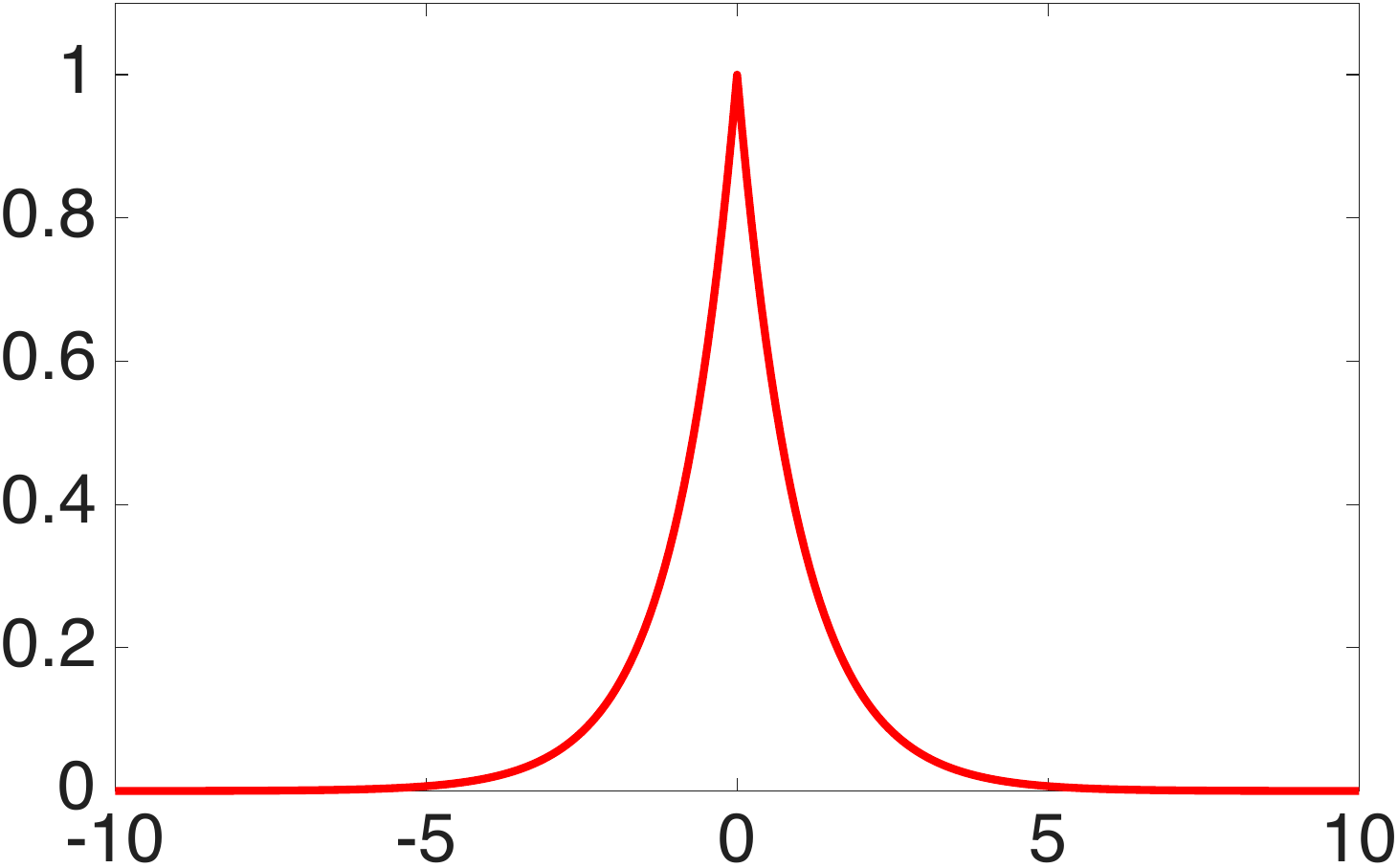}\hspace{ 0.5in}
\includegraphics[width=6cm]{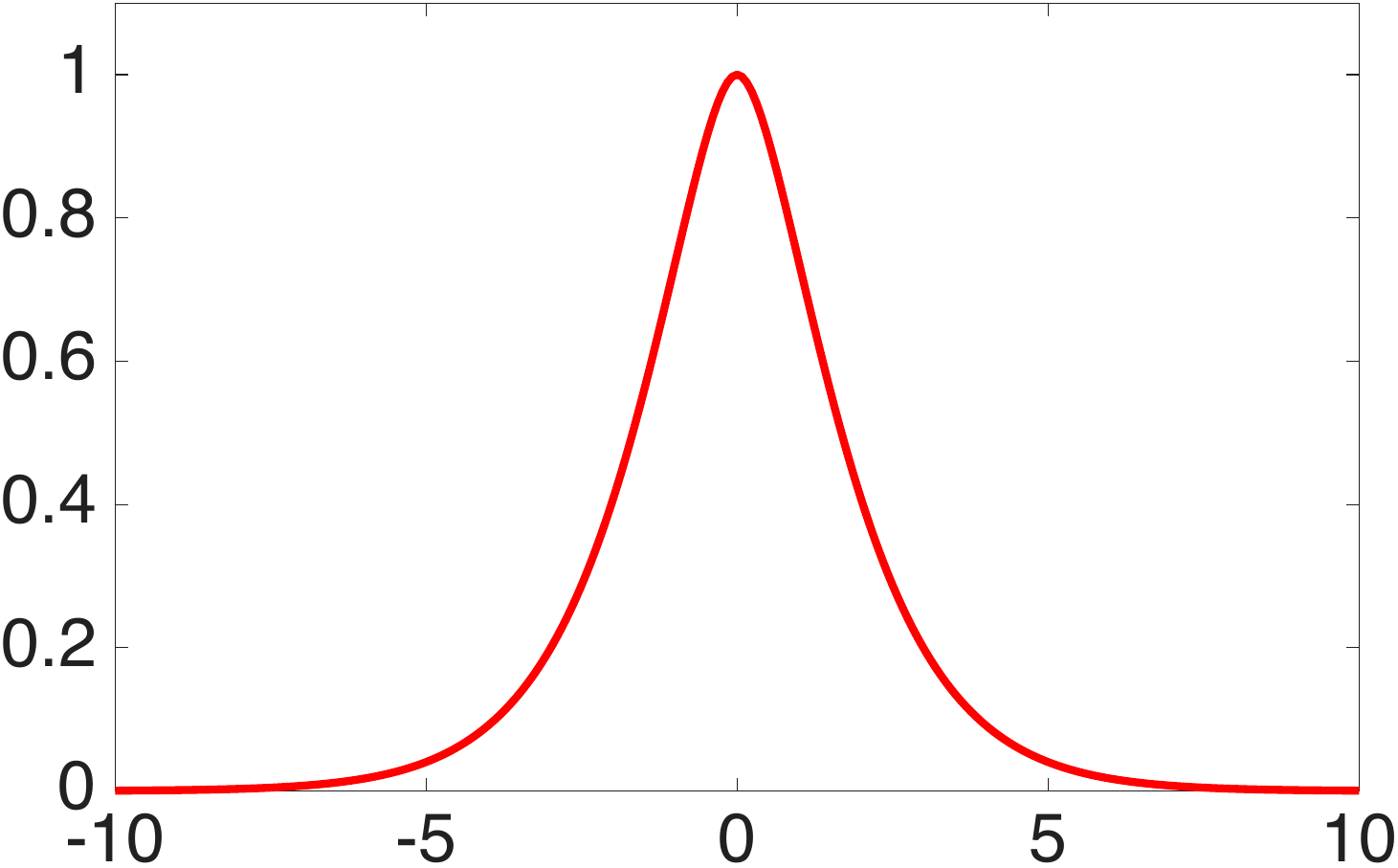}
\centering
\caption{Left: CH peakon \eqref{ch-peakon}; right: is the RZQ pseudo-peakon \eqref{qr-peakon}.}
\end{figure}

For the case where $\alpha=\beta=1$, local well-posedness was established in Sobolev spaces $H^s$ for $s>7/2$ and it was shown that solutions do not exhibit finite-time blow-up \cite{mcl1}.  For the case where $b=2$ and $m=u-u_{xx}+u_{xxxx}$, local well-posedness in Sobolev spaces $H^s$ for $s>9/2$ was shown via use of Kato's theory \cite{tian1}.  It was in \cite{rzq} and again in \cite{qr} that Qiao and Reyes extrapolated the RZQ equation \eqref{qrequation} from the FOCH equation \eqref{foch} in pursuit of other higher order Camassa-Holm type equations that generate the pseudo-peakon solutions.  The RZQ equation is completely integrable, possessing a Lax pair, and infinitely many conserved quantities \cite{qr}.  In \cite{rzq}, local well-posedness was established in Sobolev spaces $H^s(\rr)$ for $s>7/2$ using Kato's theory and finite-time blow up was obtained under certain initial conditions.  In this paper, we establish well-posedness in similar Sobolev spaces in both the periodic and non-periodic cases but by using a completely different approach that also gives us a solution size estimate utilized in our subsequent results.

\begin{theorem}\label{WP}
In both the periodic and non-periodic cases, the initial value problem for the RZQ equation is locally well-posed in $H^s$ for $s>7/2$.  
Furthermore, there exists a $T>0$ which depends only upon the size of the initial data, $\|u_0\|_{H^s}$, and $s$, such that for all $0\le t \le T$, the solution satisfies
$$
\| u(t) \|_{H^s} \le 2 \| u_0\|_{H^s}. 
$$
\end{theorem}

To establish this result, we apply a Galerkin-type approximation argument after   mollifying the term $\lambda^{-2} [ \lambda^2u \lambda^2 u_x ]$ and studying the corresponding mollified Cauchy problem.  We are then able to extract a solution utilizing the fundamental existence theorem for ordinary differential equations in Sobolev spaces $H^s$.  By further refining our mollified sequences, we may pass through a convergent subsequence that guarantees us a solution to the original ivp \eqref{QR-nl}.

We also show that the continuity of data-to-solution map is sharp by proving that it is not better than continuous.  The proof is based on the method of approximate solutions for CH-type equations in Sobolev and related spaces found in \cite{himhol, himonas2007, himonas2009euler, himonas2010cauchy, hkt, holmes2017fornbergwhitham, holmes2018huntersaxton, holmes2021rbfamily, holmes2021forq, keytig, kotz, li2020nonuniform, li2021nonuniform, li2020generalized, lyz,  thompson2013periodic2ch}.

\begin{theorem}\label{nonuniform}
In both the periodic and non-periodic cases, the data-to-solution map is not uniformly continuous from $H^s$ to $C([0, T]; H^s)$. 
\end{theorem}

The multi-pseudo peakon solutions derived in \cite{rzq, qr} combined with the argument found in \cite{byers, hhg} is sufficient to prove that the ivp for the RZQ equation is ill-posed in $H^s$ for $s<7/2$. Combined with our well-posedness result in $H^s$ for $s>7/2$, this establishes $7/2$ as the critical index of well-posedness in Sobolev spaces for the non-periodic case.
\begin{figure}[h]
\includegraphics[width=6cm]{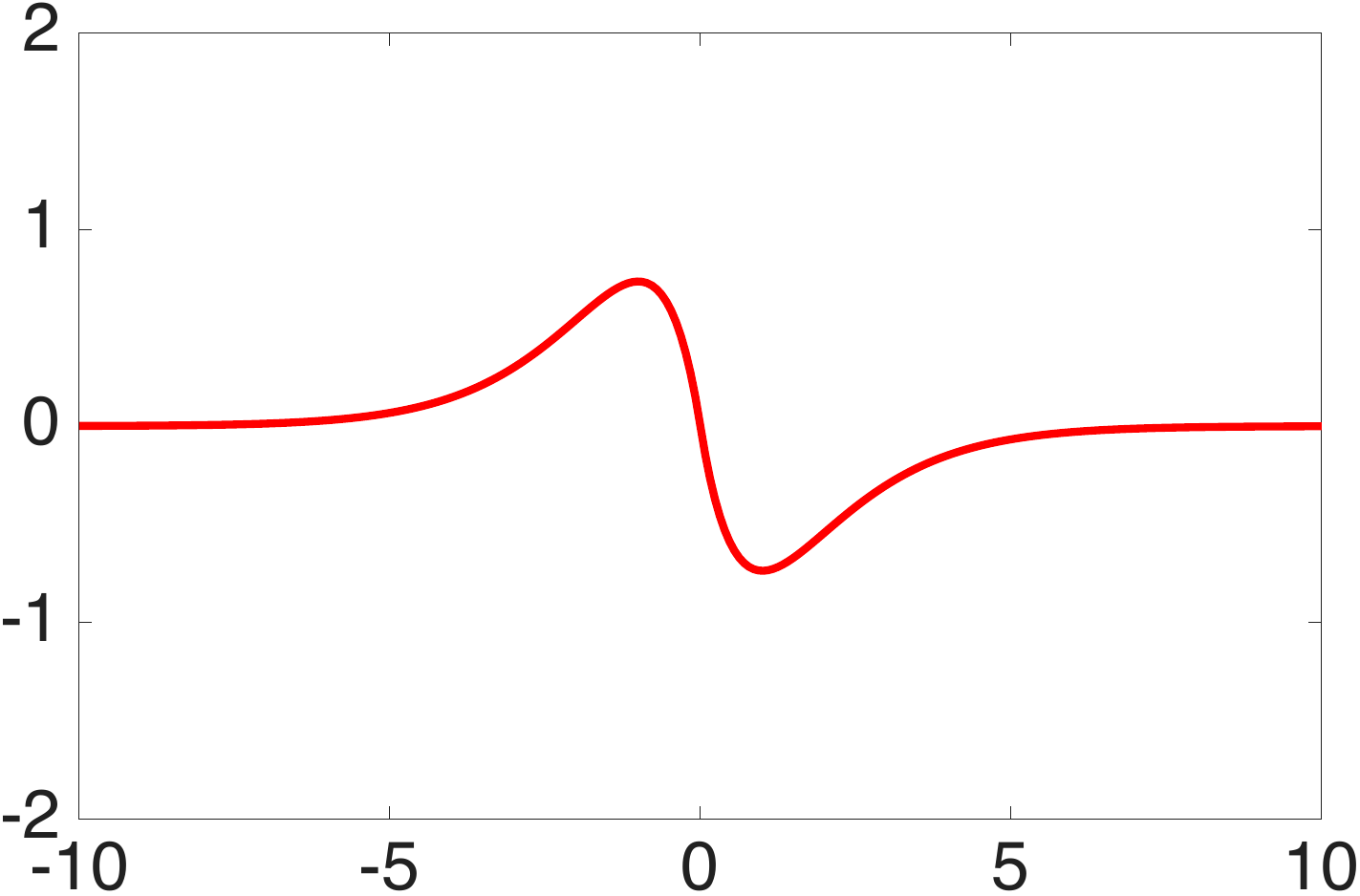}\hspace{ 0.5in}
\includegraphics[width=6cm]{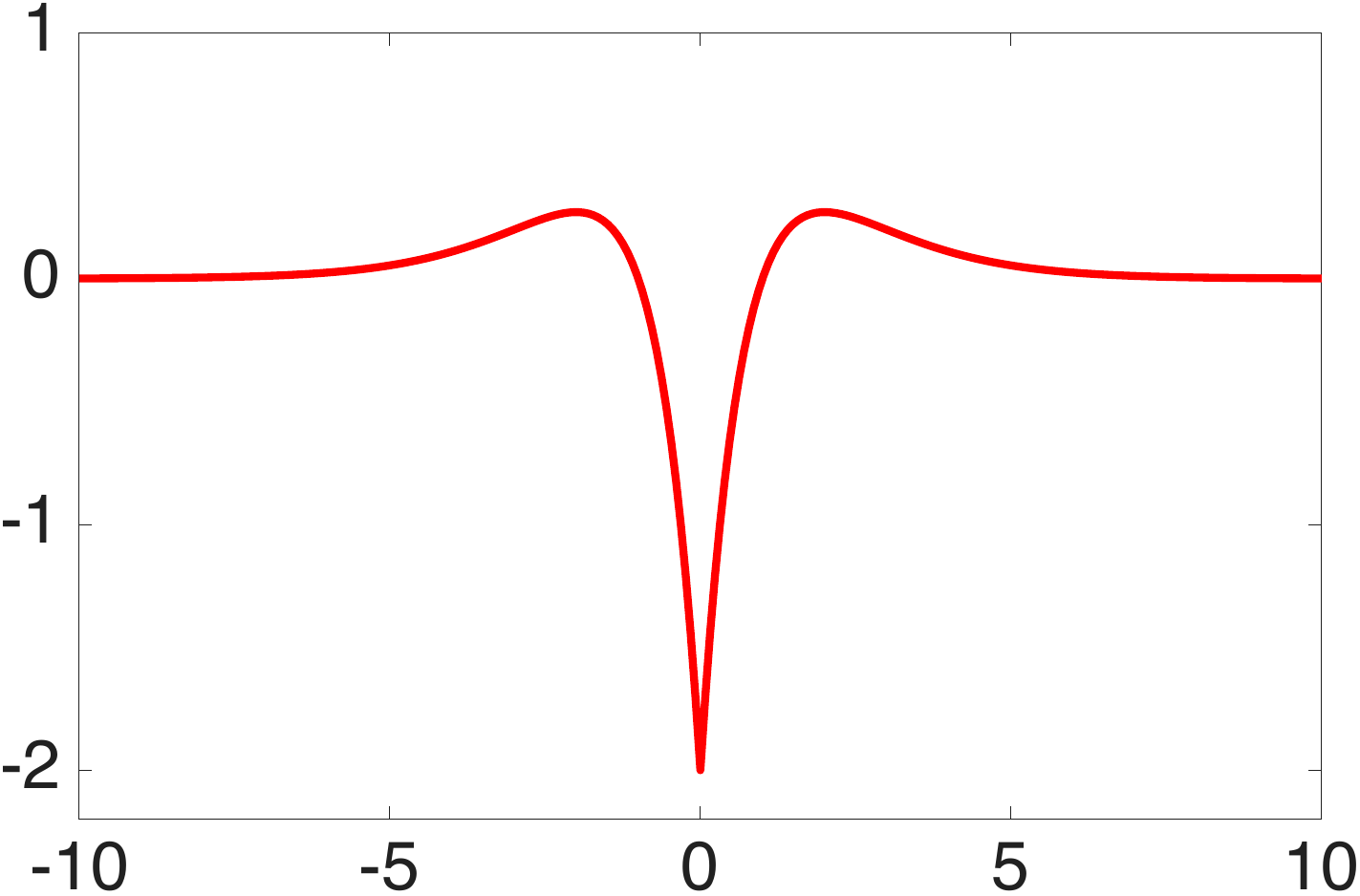}
\centering
\caption{The first (left) and second (right) derivatives of the RZQ pseudo-peakon \eqref{qr-peakon}.}
\end{figure}

\begin{theorem}\label{illp}
The RZQ equation is ill-posed in Sobolev spaces $H^s$ for $s<7/2$ on the real line.
\end{theorem}

Finally, we show the following result concerning continuity of the data-to-solution map.

\begin{theorem}
    If $s > \frac72$, $u_0 \in H^s$ and $1 \le r < s$, then the data-to-solution map corresponding to the RZQ equation  is H\"older continuous in the $H^r$ topology.
That is, for $u_0, w_0$ in the ball centered at $0$ with radius $\rho,$ the corresponding solutions satisfy 
\begin{equation}\label{holder}
    \|u(t) - w(t)\|_{H^r} \leq c\|u_0-w_0\|_{H^r}^{\alpha}, 
\end{equation}
where
\begin{equation}
    \alpha = 
\begin{cases}
    1 & \text{ if $(s,r) \in A_1$} \\
    \frac{2(s-3)}{s-r} & \text{ if $(s,r) \in A_2$} \\
    s-r & \text{ if $(s,r) \in A_3$}. 
\end{cases}
\end{equation}
Regions, $A_1, A_2,$ and $A_3$ are depicted in Figure 3.
\end{theorem}
 \begin{figure} 
 \label{figure3}
\begin{tikzpicture}[scale=0.97]
  \begin{axis}[
    xmin=-0.25, xmax=10,
    ymin=-0.25, ymax=10,
    axis lines=center,
    xlabel={$s$}, ylabel={$r$},
    xtick={0,2,4,6,8,10},
    ytick={0,2,4,6,8,10},
    enlargelimits=true,
    ticks=both,
    axis line style={->},
    clip=false
  ]
    \addplot [thick, black, smooth, domain=0:9] {x};
    \node at (axis cs:10,9.2) {\small $r = s$};

    \addplot [thick, black, smooth, domain=0:9] {x - 1};
    \node at (axis cs:10.5,8.2) {\small $r = s - 1$};

    \addplot [thick, black, smooth, domain=0:7] {6 - x}; 
    \node at (axis cs:8.4, -1.2) {\small $r = 6 - s$};

    \addplot [thick, black, dashed, domain=0:9] {1}; 
    \node at (axis cs:9.8, 1) {\small $r = 1$};
    
    \addplot [dashed, thick, black] coordinates {(3.5, -0.25) (3.5, 10)};
    \node[black] at (axis cs:3.5,10.3) {\small $s = \tfrac{7}{2}$};

    \node[black] at (axis cs:7,3) {\footnotesize $A_1$};
    \node[black] at (axis cs:5.5,5) {\footnotesize $A_3$};
    \node at (axis cs:4.1,1.4) {\footnotesize $A_2$};
    
  \end{axis}
\end{tikzpicture}
\caption{Regions $A_1$, $A_2$ and $A_3$.}
\end{figure}

Due to technical limitations, the continuity of the data-to-solution map in the $H^r$ topology  for $  r <1$ is unknown. We contrast this with earlier results by Chen, Liu and Zhang \cite{clz} for the $b$-family of equations (which includes the CH equation) or Himonas and Holmes \cite{himonasholmes} for the Novikov equation, where both the H\"older exponent is different and the region includes $r\ge 0$.  This proof employs the following new estimate which may be useful to other mathematicians.
\begin{lemma} 
If $1 \leq \mu \leq 3, \sigma - 3 > \frac12,$ and $\mu + \sigma \geq 6,$ then 
    \[\|fg\|_{H^{\mu-3}} \leq c_{\mu,\sigma} \|f\|_{H^{\sigma-3}} \|g\|_{H^{\mu-3}}.\]
\end{lemma}

Our paper is organized as follows.  In section two, we establish well-posedness in Sobolev spaces and our solution size estimate as specified in Theorem \ref{WP}.  Then in section three, we prove Theorem \ref{nonuniform} and show that the continuity of the data-to-solution map is sharp.  In section four, we prove Theorem \ref{illp} and establish ill-posedness and finally in section five we provide the proof of H\"older continuity.

\section{Proof of Well-posedness}
In this section, we shall prove Theorem \ref{WP}. As the proofs in the non-periodic case bears great similarity to the periodic case, we provide details for the periodic case and briefly outline the necessary modifications for the non-periodic setting. We refer the reader to \cite{himhol} for a similar proof with more details provided. 

We observe that for $u \in H^s$ the term $\lambda^{-2} [ \lambda^2 u \lambda^2 u_x ]  \in H^{s-1}$, $s>7/2$. Replacing this term with the mollified smooth version, we obtain an ODE on the space $H^s$
\begin{align}\label{mollifiedQR}
 & u_t  +J_\ee \lambda^{-2} [ \lambda^2 J_\ee u \lambda^2\partial_x J_\ee u ] 
  = 
  - \lambda^{-4} [   \lambda^2 u_x \lambda ^2u_{xx}  + 2 
 \lambda ^2  u \lambda^2 u_x   ]  ,
 \\
\label{mollifiedQRdata}
& u_\ee(x,0) =  J_\ee u_0(x).
\end{align}
Here, for each $\ee \in (0, 1]$, the operator $J_\ee u = j_\ee * u$ is the well-known Friedrichs mollifier. The function $j_\ee$ is defined by first fixing a Schwartz function $j(x)\in\mathcal{S}(\mathbb{R})$ and then define the periodic functions $j_\ee$ by
\[
j_\ee(x)\doteq\frac{1}{2\pi}\sum_{n\in\mathbb{Z}}\widehat{j}(\ee n)e^{inx}
\]
satisfying $0\leq\widehat{j}(\xi)\leq1$ for all $\xi\in\mathbb{R}$, and $\widehat{j}(\xi)=1$ for all $\xi \in [-1,1]$.

We also define the function
\[
F_\ee(u) = -J_\ee \lambda^{-2} [ \lambda ^2J_\ee u \lambda ^2\partial_x J_\ee u ] 
  - \lambda^{-4} [   \lambda^2 u_x \lambda^2 u_{xx}  + 2 
 \lambda ^2  u \lambda^2 u_x   ]
\]
and note that
\begin{align*}
\p_uF_\ee(u)h =&   -J_\ee[\lambda^2\p_xJ_\ee u\lambda^2 J_\ee h+\lambda^2 J_\ee u \lambda ^2\p_x J_\ee h] 
\\ &  - \lambda^{-4}[\lambda^2\p_xu\lambda^2 \p_x^2h+\lambda^2\p_x^2u\lambda^2\p_xh+2\lambda^2 u\lambda^2\p_x h+2\lambda^2\p_xu\lambda^2 h].
\end{align*}

Therefore, for each $\ee$, the mollified smooth version, equation \eqref{mollifiedQR} defines an ODE on $H^s$, and thus has a unique solution $u_\ee$ with lifespan $T_\ee$. We shall first prove that there exists a uniform lifespan for the solutions to the ivp \ref{mollifiedQR}-\ref{mollifiedQRdata}. 
\begin{theorem}
For each $\ee \in (0,1]$, there exists a $T $ which depends only on the size of the initial data, $\|u_0\|_{H^s}$, and a constant $c_s$, which depends only on $s>7/2$, such that for all $0<|t|<T$, the solution $u_\ee$ to the ivp \ref{mollifiedQR}-\ref{mollifiedQRdata} satisfies
$$
\|u_\ee (t) \|_{H^s} \le \frac{\|J_\ee u_0\|_{H^s}}{1-c_s\|J_\ee u_0\|_{H^s} t}.
$$
\end{theorem}
\begin{proof}
    Apply $ \lambda^{s}$ on the left side, multiply by $\lambda^{s} u_\ee$ on the right hand side, and integrate to find 
\begin{align}
   \frac{d}{dt}  \|u_\ee\|^2_{H^s}   
   =  &
  - \int \lambda^{s-2} J_\ee \left [ \lambda^2 J_\ee u_\ee \lambda^2 \partial_x J_\ee u _\ee \right]  \lambda^s u_\ee dx
\notag
\\
  &- \int \lambda^{s-4} \partial_x \left[    \frac12( \lambda^2 \partial_x u_\ee  )^2 +  
(  \lambda^2   u _\ee )^2  \right]  \lambda^s u_\ee dx .
\notag
\end{align}
  We apply the Cauchy-Schwarz inequality and the algebra property for Sobolev spaces to the second term  to obtain 
  \begin{align}
    \frac{d}{dt}  \|u_\ee\|^2_{H^s}   
  & \le
  - \int \lambda^{s-2} J_\ee \left [ \lambda^2 J_\ee u_\ee \lambda^2 \partial_x J_\ee u _\ee \right]  \lambda^s u_\ee dx
  +     C \|u_\ee\|_{H^s}^3. 
\end{align}
For the first term, we will use the following commutator estimate found in \cite{kp}.
\begin{lemma}\label{KP} [Kato-Ponce] If $s>0$ then there is a $c_s>0$ such that 
\begin{align}
\| \lambda^s [fg]- f\lambda^s[g]\|_{L^2} \le c_s \left( \|\lambda^s f\|_{L^2} \|g\|_{L^\infty}+ \|\partial_x f\|_{L^\infty} \|\lambda^{s-1}g\|_{L^2} 
\right). 
\end{align}
\end{lemma}

We commute the operator $\lambda^s$ with $J_\ee u_\ee$ and apply Lemma \ref{KP}, use the Sobolev embedding theorem and the algebra property. The result is  
  \begin{align}
   \frac{d}{dt}  \|u_\ee\|^2_{H^s}   
  & \le
    -  \int  \lambda^2 u_\ee \lambda^s \partial_x  u_\ee    \lambda^s u_\ee dx
  +     C \|u\|_{H^s}^3. 
\end{align}
 For the remaining integral, we commute $\partial_x$ with $\lambda^s$, and then apply integration by parts followed by H\"older's inequality. The result is that for a constant depending only on $s>7/2$ we have 
  \begin{align} 
   \frac{d}{dt}  \|u_\ee\|^2_{H^s}  
  & \le C \|u_\ee\|_{H^s}^3 . 
\end{align}
Solving this differential inequality gives the desired result. 
\end{proof}

By taking $t \leq T = 1/(2c_s\|u_0\|_{H^s})$, we find that $u_\ee$ satisfies the solution size bound
\begin{equation} \label{soln-bd}
    \|u_\ee\|_{H^s} \leq 2\|J_\ee u_0\|_{H^s} \leq 2\|u_0\|_{H^s}.
\end{equation}
Furthermore, we have that
\begin{equation}
    \|\p_tu_\ee\|_{H^{s-1}} \lesssim \|u_\ee\|_{H^s}^2 \lesssim \|u_0\|_{H^s}^2.
\end{equation}
In order to prove that the ivp for the RZQ equation is well-posed, we shall show that there exists a subsequence of $u_\ee$, $\ee \rightarrow 0$, that converges to a solution $u$ of the ivp for the RZQ equation. We will then show that solutions are unique and depend continuously upon the initial data, $u_0$.

\subsection{Existence of Solutions} Our first step is showing that there exists a function, $u$, constructed as the limit of a subsequence of the family of solutions $u_\ee$, as $\ee \rightarrow 0$. In order to find this subsequence, we shall refine the sequence $\{u_\ee\}$ several times labeling the subsequence $\{u_{\ee_\nu}\}$, and after each refinement, we shall for convenience, relabel the subsequence $\{u_\ee\}$. 

\subsubsection{Weak$^*$ convergence in $L^\infty([0,T]; H^s)$}Observing that the family, $\{u_\ee\}$ is bounded in\\ $C([0, T]; H^s) \subset L^\infty([0, T]; H^s) $, the Banach-Alaoglu theorem tells us that $\{u_\ee\}$  is precompact with respect to the weak$^*$ topology in $\bar{B}(0,2\|u_0\|_{H^s}) \subset L^\infty([0, T]; H^s)$. Therefore there is an element $u \in  \bar{B}(0,2\|u_0\|_{H^s})$ and a subsequence $\{u_{\ee_\nu}\}$ which converges to $u$ weakly$^*$. 

\subsubsection{Strong convergence in $C([0,T]; H^{s-1})$}We now relabel the sequence $\{u_{\ee_\nu}\}$ as $\{u_\ee\}$ and show that this sequence converges strongly to $u$ in the $C([0,T];H^{s-1})$ topology.  Indeed, we note that for $t_1,t_2 \in [0,T]$ and the Mean Value Theorem, we have that
\[
\|u_\ee(t_1)-u_\ee(t_2)\|_{H^{s-1}} \leq \sup_{t\in[0,T]}\|\p_tu_\ee\|_{H^{s-1}}|t_1-t_2| \lesssim |t_1-t_2|.
\]
This establishes the fact that the family $\{u_\ee\}$ is equicontinuous.  Furthermore, since the torus is a compact manifold, the inclusion mapping $i:H^s \hookrightarrow H^{s-1}$ is a compact operator, and therefore we may deduce that $U(t) = \{u_\ee(t)\}$ is a precompact set in $H^{s-1}$.  Thus, by Ascoli's theorem, there is a subsequence $\{u_{\ee_\nu}\}$ converging to an element in $H^{s-1}$.  By uniqueness of the limit, this must be $u$.

\subsubsection{Strong convergence in $C([0,T]; H^{s-\s})$}We now wish to demonstrate that this sequence $\{u_{\ee_\nu}\}$ (relabeled as $\{u_\ee\}$) converges strongly to $u$ in the space $C([0,T];H^{s-\s})$ for $\s \in (0,1)$.  To accomplish this task, we will need the following results.

\begin{lemma}
For $\ee \in [0,1)$ and $\s \in (0,1)$ we have $u_\ee \in C^\s([0,T];H^{s-\s})$.  Furthermore, the $C^\s([0,T];H^{s-\s})$ norm of $u_\ee$ is bounded by
\[
\|u_\ee\|_{C^\s([0,T];H^{s-\s})} \lesssim \|u_0\|_{H^s}+\|u_0\|_{H^s}^2.
\]
\end{lemma}

\begin{proof}
    By definition, we have
    \[
    \|u_\ee\|_{C^\s([0,T];H^{s-\s})} \doteq \sup_{t \in [0,T]}\|u_\ee\|_{H^{s-\s}}+\sup_{t \neq t'}\frac{\|u_\ee(t)-u_\ee(t')\|_{H^{s-\s}}}{|t-t'|^{\s}}.
    \]
  Not that the first term on the right hand side may be bounded by using the solution size estimate.  The second term may be handled by noting the inequality $x^\s \leq 1+x$ for $x>0$ where $x = 1/((1+k^2)|t-t'|^2$ and we see that
  \begin{align*}
      \sup_{t \neq t'}\frac{\|u_\ee(t)-u_\ee(t')\|_{H^{s-\s}}^2}{|t-t'|^{2\s}} & = \sup_{t \neq t'}\sum_{k \in \zz}(1+k^2)^{s-\s}\frac{|\wh{u_\ee}(t)-\wh{u_\ee}(t')|^2}{|t-t'|^{2\s}} \\ 
& \leq \sup_{t \neq t'}\sum_{k \in \zz}(1+k^2)^s|\wh{u_\ee}(t)-\wh{u_\ee}(t')|^2+\sup_{t \neq t'}\sum_{k \in \zz}(1+k^2)^{s-1}\frac{|\wh{u_\ee}(t)-\wh{u_\ee}(t')|^2}{|t-t'|^{2}} \\
 &\leq \|u_\ee\|_{C([0,T];H^s)}^2+\|\p_tu_\ee\|_{C([0,T];H^{s-1})}^2 \\ 
 & \lesssim \|u_0\|_{H^s}^2+\|u_0\|_{H^s}^4
  \end{align*}
  By taking square roots of both sides we obtain the desired result.
\end{proof}

We again find that the sequence is equicontinuous in $H^{s-\s}$ since
\[
\|u_\ee(t)-u_\ee(t')\|_{H^{s-\s}} \leq \|u_\ee\|_{C^\s([0,T];H^{s-\s})}|t-t'|^\s.
\]
Using the fact that the inclusion mapping $i:H^s \hookrightarrow H^{s-\s}$ is a compact operator, the sequence $\{u_\ee\}$ is precompact in $H^{s-\s}$ and hence, by Ascoli's theorem, there exists a subsequence $\{u_{\ee_\nu}\}$ that converges to an element in that space.  By uniqueness of the limit, that element must be $u$.

\subsubsection{Convergence in $C^3$}We now choose $\s>0$ so small that $s-\s>7/2$ and apply the Sobolev embedding theorem to obtain
\[
\|u_\ee(t)-u_\ee(t')\|_{C^3} \lesssim \|u_\ee(t)-u_\ee(t')\|_{H^{s-\s}}.
\]
Therefore, we may also find that the sequence $\{u_\ee\}$ is also equicontinuous in $C^3$ since
\[
\|u_\ee(t)-u_\ee(t')\|_{C^3} \lesssim \|u_\ee(t)-u_\ee(t')\|_{H^{s-\s}} \lesssim |t-t'|^\s.
\]
Furthermore, we have that
\[
\|u_\ee\|_{H^{s-\s}} \lesssim \|u_\ee\|_{H^s} \lesssim \|u_0\|_{H^s}.
\]
Therefore, by Rellich's lemma, for each $t \in [0,T]$, the set $\{u_\ee\}$ is precompact in $H^{s-\s}$.  That is, any subset of $\{u_\ee\}$ contains a sequence $\{u_{\ee_n}\}$ that converges to an element $u \in H^{s-\s}$.  By the above estimates, we may conclude that $\{u_\ee\}$ is precompact in $C^3$.  Therefore, we find that $\{u_{\ee_\nu}\} \to u$ in $C([0,T];C^3)$.

\subsubsection{Verifying that the limit $u$ solves the RZQ equation}  Here, we will need the following lemma.
\begin{lemma}
\label{linop}
For $0<r\leq s$, the map $I-J_\ee:H^s \rightarrow H^r$ satisfies the norm operator estimate 
\begin{equation*}
\|I-J_\ee\|_{\mathcal{L}(H^s,H^r)}=o(\ee^{s-r}).
\end{equation*}
\end{lemma}
We will now show that the sequence $\{\p_tu_{\ee_\nu}\}$ (relabeled as $\{\p_tu_\ee\}$) converges to $\p_tu$ in the space $C([0,T];C^2)$.  Starting from our mollified RZQ \eqref{mollifiedQR}, we have
\begin{equation}\label{molli-de}
    \p_tu_\ee = -J_\ee \lambda^{-2} [ \lambda^2 J_\ee u_\ee \lambda^2 \partial_x J_\ee u_\ee ] 
  - \lambda^{-4} \p_x \left[ \frac12(\lambda^2 \p_xu_\ee)^2  +  
 (\lambda ^2  u_\ee)^2   \right].
\end{equation}
By continuity of the operator $\lambda^{-2}\p_x$, we may immediately deduce the convergence of the nonlocal terms
\begin{align*}
    &\frac12\lambda^{-4}\p_x(\lambda ^2\p_xu_\ee)^2 \to \frac12\lambda^{-4}\p_x(\lambda^2 \p_xu)^2, \\
    &\lambda^{-4}\p_x(\lambda^2 u_\ee)^2 \to \lambda^{-4}\p_x(\lambda^2 u)^2.
\end{align*}
To handle the mollified twisted-Burgers term $J_\ee \lambda^{-2} [ \lambda^2 J_\ee u_\ee \lambda^2 \partial_x J_\ee u_\ee ]$ in \eqref{molli-de} we first rewrite it as $(1/2)J_\ee\lambda^{-2}\p_x(\lambda^2 J_\ee u_\ee)^2$ and then apply triangle inequality to obtain
\begin{align}\label{ineq1}
\|J_\ee\lambda^{-2}\p_x(\lambda^2 J_\ee u_\ee)^2-\lambda^{-2}\p_x(\lambda^2 u)^2 & \|_{C([0,T];C^2)}  \leq \|J_\ee\lambda^{-2}\p_x(\lambda ^2J_\ee u_\ee)^2-\lambda^{-2}\p_x(\lambda^2  J_\ee u_\ee)^2\|_{C([0,T];C^2)} \nonumber \\
&+\|\lambda^{-2}\p_x(\lambda^2 J_\ee u_\ee)^2-\lambda^{-2}\p_x(\lambda ^2u)^2\|_{C([0,T];C^2)}
\end{align}
For the first term, choose $5/2<r<s-\s-1<s-1$.  Then we have by Lemma \ref{linop}
\begin{align}\label{ineq2}
\|J_\ee\lambda^{-2}\p_x(\lambda^2J_\ee u_\ee)^2-\lambda^{-2}\p_x(\lambda^2  J_\ee u_\ee)^2\|_{C([0,T];C^2)} &\lesssim \|J_\ee\lambda^{-2}\p_x(\lambda ^2J_\ee u_\ee)^2-\lambda^{-2}\p_x(\lambda^2  J_\ee u_\ee)^2\|_{H^r} \nonumber \\ 
& \lesssim \|I-J_\ee\|_{\mathcal{L}(H^{s-1},H^r)}\|\lambda^{-2}\p_x(\lambda ^2J_\ee u_\ee)^2\|_{H^{s-1}} \nonumber \\ 
& = o\left(\ee^{s-r-1}\right).
\end{align}
For the second term, we have that
\begin{align}\label{ineq3}
    \|\lambda^{-2}\p_x(\lambda^2 J_\ee u_\ee)^2-\lambda^{-2}\p_x(\lambda ^2u)^2\|_{C([0,T];C^2)} &\lesssim \|\lambda^{-2}\p_x(\lambda ^2J_\ee u_\ee)^2-\lambda^{-2}\p_x(\lambda^2 u)^2\|_{H^r} \nonumber \\
    & \lesssim \|J_\ee u_\ee + u\|_{H^{r+1}}\|J_\ee u_\ee - u\|_{H^{r+1}} \nonumber \\ 
    &\lesssim \|J_\ee u_\ee - u\|_{H^{r+1}} \nonumber \\ 
    &\leq \|J_\ee u_\ee-u_\ee\|_{H^{r+1}}+\|u_\ee-u\|_{H^{r+1}}.
\end{align}
For the first term in \eqref{ineq3}, we use Lemma \ref{linop} to obtain
\begin{equation}\label{ineq4}
    \|J_\ee u_\ee-u_\ee\|_{H^{r+1}} \lesssim \|I-J_\ee\|_{\mathcal{L}(H^s,H^{r+1})}\|u_\ee\|_{H^s} = o\left(\ee^{s-r-1}\right).
\end{equation}
For the second term, we note that
$
\|u_\ee-u\|_{H^{r+1}} \lesssim \|u_\ee-u\|_{H^{s-\s}} \to 0
$
as $\ee \to 0$.  Thus, we conclude that $\{\p_tu_\ee\} \to \p_t u$ in $C([0,T];C^2)$.

\subsubsection{Improving the regularity to $C([0,T];H^s)$}  Now that we know there is a solution \\ $u \in L^\infty([0,T]; H^s)$, we wish to improve it's regularity and show that $u \in C([0,T];H^s)$.  Therefore, we will fix $t \in [0,T]$ and consider a sequence $\{t_n\} \to t$.  We will show that $\|u(t_n) - u(t)\|_{H^s} \to 0$ as $n \to \infty$.  To do so, we establish the following lemmas.
\begin{lemma}\label{inprod}
    The solution $u \in L^\infty([0,T]; H^s)\cap Lip([0,T];H^{s-1})$ is continuous in $t$ with respect to the weak topology on $H^s$.  That is, for any $\varphi \in H^s$, and a sequence $\{t_n\} \to t$, we have that $$\langle u(t_n)-u(t),\varphi\rangle_{H^s} \to 0 \ \ \ \text{as} \ \ \{t_n\} \to t.$$
\end{lemma}
\begin{proof}
    Let $\ee >0$.  We will choose a $\psi \in \mathcal{S}(\rr)$ with $\|\vp -\psi\|_{H^s} < \ee/8\|u_0\|_{H^s}$.  By use of the triangle inequality, we have that
    \begin{equation} \label{inprodest}
        |\langle u(t_n)-u(t),\varphi\rangle_{H^s}| \leq |\langle u(t_n)-u(t),\varphi -\psi\rangle_{H^s}|+|\langle u(t_n)-u(t),\psi\rangle_{H^s}|.
    \end{equation}
    For the first term in \eqref{inprodest}, we have 
    \begin{equation}\label{inprodest1}
        |\langle u(t_n)-u(t),\varphi -\psi\rangle_{H^s}| \leq (\|u(t_n)\|_{H^s}+\|u(t)\|_{H^s})\|\varphi-\psi\|_{H^s}<\ee/2.
    \end{equation}
    For the second term, we find the following estimate.
    \begin{align}\label{inprodest2}
        |\langle u(t_n)-u(t),\psi\rangle_{H^s}| &\leq \|u(t_n) - u(t)\|_{H^{s-1}}\|\psi\|_{H^{s+1}} \nonumber \\
        &\lesssim \|\p_tu(t)\|_{H^{s-1}}\|\psi\|_{H^{s+1}}|t_n - t| \nonumber \\
        &\lesssim\|\psi\|_{H^{s+1}}|t_n-t|.
    \end{align}
    By choosing $n$ sufficiently large, we may bound \eqref{inprodest2} by $\ee/2$ so that we achieve $$|\langle u(t_n)-u(t),\varphi\rangle_{H^s}|<\ee,$$
    which completes the proof.
\end{proof}

\begin{lemma}
    For the solution $u \in L^\infty([0,T]; H^s)\cap Lip([0,T];H^{s-1})$, we have $$\lim_{n \to \infty}\|u(t_n)\|_{H^s} = \|u(t)\|_{H^s}.$$
\end{lemma}
\begin{proof}
    We begin by defining the functions 
    \begin{equation}
        F(t) \doteq \|u(t)\|_{H^s} \ \ \text{and} \ \ F_\ee(t) \doteq \|J_\ee u(t)\|_{H^s}.
    \end{equation}
    By Lemma \ref{linop}, we have that $F_\ee \to F$ pointwise as $\ee \to 0$.  We will prove that each $F_\ee$ is Lipschitz, and that the Lipschitz constants for this family of functions are uniformly bounded by some independent constant.  This will imply that $F$ is also Lipschitz continuous and its Lipschitz constant is bounded by this same independent constant which will complete our proof.  Applying both sides of \eqref{QR-nl} by the operator $J_\ee$, we obtain
    \begin{equation}
        \p_tJ_\ee u = -J_\ee\lambda^{-2} (\lambda^2 u \lambda^2 u_x)-J_\ee\lambda^{-4}[\lambda^2 u_x \lambda^2 u_{xx}+2\lambda^2 u \lambda^2 u_x].
    \end{equation}
    Proceeding in the same manner as section two on energy estimates, one may obtain the following $$\frac{d}{dt}\|J_\ee u(t)\|_{H^s} \lesssim \|u(t)\|_{H^s}^2.$$
    Using our lifespan estimate, we obtain $$|F_\ee'(t)| = \left|\frac{d}{dt}\|J_\ee u(t)\|_{H^s}\right|\leq c_s\|u_0\|_{H^s}^2.$$ 
    Thus, $F(t)$ is Lipschitz.  This concludes our proof.
\end{proof}

Now, observe that
\begin{equation}
    \|u(t_n) - u(t)\|_{H^s}^2 = \|u(t_n)\|_{H^s}^2+\|u(t)\|_{H^s}^2 - \langle u(t_n),u(t)\rangle_{H^s} - \langle u(t),u(t_n)\rangle_{H^s}.
\end{equation}
Using Lemma \ref{inprod} with $\varphi = u$, we have that $\lim_{n \to \infty}\langle u(t_n),u(t)\rangle=\|u(t)\|_{H^s}^2$.  Coupled with our last lemma, we find that $$\|u(t_n) - u(t)\|_{H^s} \to 0 \ \ \text{as} \ \ n \to \infty.$$  Hence, $u \in C([0,T]; H^s).$

\subsection{Uniqueness of Solutions} Our second step in our proof of well-posedness is to prove that solutions are unique. We shall accomplish this by providing an energy estimate on the difference between two solutions in $H^\sigma$, $5/2<\sigma<s-1$. Indeed, we let $u$ and $v$ be two solutions to equation \eqref{QR-nl} corresponding to initial data $u(x,0)=u_0(x)$ and $v(x,0)=v_0(x)$ respectively. 
That is to say,
\begin{equation} 
 \begin{cases} u_t  + \lambda^{-2} \left[ \lambda ^2u \lambda^2 \partial_x u   \right] 
  = 
  - \lambda^{-4}  \partial_x \left[  \frac12 \left(\lambda ^2\partial_x  u  \right )^2  +   
\left( \lambda ^2  u  \right )^2  \right ]   
\\
u(x,0) = u_0(x) ,
\end{cases}
\end{equation}
and
\begin{equation} 
 \begin{cases} 
  v_t  + \lambda^{-2} \left[ \lambda^2 v \lambda^2\partial_x v  \right ] 
  = 
  - \lambda^{-4}  \partial_x \left[  \frac12 \left(\lambda^2 \partial_x  v  \right )^2  +   
\left( \lambda ^2  v  \right )^2 \right  ]   
\\
v(x,0) = v_0(x). 
\end{cases}
\end{equation}
Subtracting $v$ from $u$ and setting $w = u-v$, we observe that $w$ is a solution to 
\begin{equation} 
 \begin{cases} w_t  +  \frac12 \lambda^{-2}\partial_x  \left[ (\lambda^2 u +\lambda^2 v) \lambda ^2  w   \right] 
  = 
  - \lambda^{-4}  \partial_x \left[  \frac12  (\lambda^2\partial_x u +\lambda ^2\partial_xv) \lambda   ^2\partial_x w  +   
  (\lambda^2 u +\lambda^2 v) \lambda^2   w     \right ]   
\\
w(x,0) = u_0(x) -v_0(x)
\end{cases}  
 \end{equation}

For $5/2<\sigma<s-1 $
 fixed, we calculate the $H^\sigma$ energy of $w$ giving us the equation 
 \begin{align}    
 \frac{d}{dt} \| w\|_{H^\sigma}^2 =& -   \frac12 \int \lambda^\sigma \left[ \lambda^{-2}\partial_x  \left[ (\lambda^2 u +\lambda^2 v) \lambda^2   w   \right]  \right]   \lambda^\sigma w dx 
     \label{firstterm-w}
 \\ & 
  - \int \lambda^\sigma \lambda^{-4}  \partial_x \left[  \frac12  ( \lambda^2 \partial_x u +\lambda^2  \partial_xv) \lambda^2  \partial_x w   \right ]    \lambda^\sigma w dx
    \label{secondterm-w}
  \\ &
 - \int \lambda^\sigma \lambda^{-4}  \partial_x \left[   
  (\lambda^2 u +\lambda^2v)\lambda^2   w     \right ]    \lambda^\sigma w dx .
  \label{thirdterm-w}
\end{align}
For the third integral, \eqref{thirdterm-w}, we apply the Cauchy-Schwarz inequality followed by the Sobolev embedding theorem to find 
 \begin{align}    \notag 
 \int \lambda^\sigma \lambda^{-4}  \partial_x \left[   
  (\lambda^2 u +\lambda^2v)\lambda^2   w     \right ]    \lambda^\sigma w dx  
 \le &\
  C 
  \left\| (\lambda^2 u +\lambda^2v)\lambda^2   w
  \right\| _{H^{\sigma-2} } 
  \|w\|_{H^\sigma} .
\end{align}
The algebra property followed by the solution size estimate allows us to bound this by 
 \begin{align}    \notag 
 \int \lambda^\sigma \lambda^{-4}  \partial_x \left[   
  (\lambda^2 u +\lambda^2v)\lambda^2   w     \right ]    \lambda^\sigma w dx  
 \le & \
  C 
 \left(  \left\| u_0   \right\| _{H^{s} }  +  \left\| v_0   \right\| _{H^{s} } \right)  
  \|w\|_{H^\sigma} ^2 .
\end{align}

For the second integral, the one found in \eqref{secondterm-w}, we again apply the Cauchy-Schwarz inequality to find 
 \begin{align}    
 \eqref{secondterm-w}\ \dot=  & \ 
   \frac12  \int \lambda^{\sigma-4}    \partial_x \left[    ( \lambda^2 \partial_x u +\lambda^2  \partial_xv) \lambda^2  \partial_x w   \right ]    \lambda^\sigma w dx 
  \\  \le & \
 C \|( \lambda^2 \partial_x u +\lambda^2  \partial_xv) \lambda^2  \partial_x w  \|_{H^{\sigma - 3} }
 \| w  \|_{H^{\sigma } }. 
\end{align}
We now apply the following estimate found in Lemma 2 of \cite{himonasholmes}. 
\begin{lemma}\label{HH}
    If $r>1/2$, then 
    $$ \| fg\|_{H^{r-1}} \le c_r \| f\|_{H^r}\|g\|_{H^{r-1}}. 
    $$
\end{lemma}
Applying the lemma with $r = \sigma - 2>1/2$ we obtain 
 \begin{align}    
 \eqref{secondterm-w} 
    \le & \ 
 C \|( \lambda^2 \partial_x u +\lambda^2  \partial_xv) \|_{H^{\sigma - 2} }  \|  \lambda^2  \partial_x w  \|_{H^{\sigma - 3} }
 \| w  \|_{H^{\sigma } }. 
\end{align}
Using the solution size estimate and $s > \sigma +1 $, we have 
 \begin{align}    
 \eqref{secondterm-w} 
    \le & \ C \left(  \left\| u_0   \right\| _{H^{s} }  +  \left\| v_0   \right\| _{H^{s} } \right)  
  \|w\|_{H^\sigma} ^2 . 
\end{align}
To bound the first term, \eqref{firstterm-w}, we shall use the following Calderon-Coifman-Meyer type commutator estimate, which follows from proposition 4.2 in Taylor \cite{taylor}
\begin{lemma}\label{ccm}
    If $r>3/2$ and $0\le \rho+1 \le r $, then 
    $$
    \| \lambda^\rho\partial_x [fv] - f \lambda^\rho\partial_x [v]\|_{L^2 } \le C \| f\|_{H^r } \| v\|_{H^\rho }.
    $$
\end{lemma}

We let $f = \lambda^2 u +\lambda^2 v$, thus we have 
 \begin{align}    
 \eqref{firstterm-w} 
    \ \dot = & \ 
     -   \frac12 \int \lambda^{\sigma-2 } \partial_x  \left[ f \lambda^2   w   \right]     \lambda^\sigma w dx 
    \\
    \le & \ 
    \left|  \int  \left(  \lambda^{\sigma-2 } \partial_x  \left[ f \lambda^2   w   \right] - f \lambda^{\sigma-2 } \partial_x   \lambda^2    w      \right)  \lambda^\sigma w dx  \right| 
    +    \left|  \int  f \lambda^{\sigma } \partial_x       w   \lambda^\sigma w dx 
    \right| 
    . \label{firstterm-ce}
\end{align}
For the first term in inequality \eqref{firstterm-ce}, we apply the Cauchy-Schwarz inequality followed by Lemma \ref{ccm} with $\rho = \sigma-2$ and $r =s-2$ to obtain 
$$
 \left|  \int  \left(  \lambda^{\sigma-2 } \partial_x  \left[ f \lambda^2   w   \right] - f \lambda^{\sigma } \partial_x       w      \right)  \lambda^\sigma w dx  \right| 
 \le C\|f\|_{H^{s-2}}\|\lambda^2 w\|_{H^{\sigma -2}} \|w\|_{H^\sigma} \lesssim \| w\|_{H^\sigma }^2.
$$
For the second term in inequality \eqref{firstterm-ce}, we integrate by parts to find 
$$
  \left|  \int  f \lambda^{\sigma } \partial_x       w   \lambda^\sigma w dx 
    \right|  
    \approx
      \left|  \int  \partial_x   f (\lambda^{\sigma }         w )^2   dx 
    \right| 
    \lesssim \| \partial_x f \|_{L^\infty}  \| w\|_{H^\sigma }^2.
$$
Since $s>7/2$, by the Sobolev embedding theorem and the solution size estimate, we have 
$ \| \partial_x f \|_{L^\infty}  \lesssim \|u_0\|_{H^s } + \|v_0\|_{H^s }$. 

Putting our estimates for \eqref{firstterm-w}, \eqref{secondterm-w}, and \eqref{thirdterm-w} together, we have proved the following lemma. 
\begin{lemma}
    Let $s>7/2$ and $5/2<\sigma<s-1$. If $u$ and $v$ are two solutions to the RZQ equation on the time interval $[0, T]$ corresponding to initial data $u_0, \ v_0 \in H^s$ respectively, then their difference $w = u-v$ satisfies the following differential inequality 
    \begin{align}
       \frac{d}{dt} \| w\|_{H^\sigma}^2 \le C\| w\|_{H^\sigma}^2
     \quad \text{ with } \quad    \|w(0)\| _{H^\sigma}^2=     \|u_0-v_0\| _{H^\sigma}^2.
    \end{align}
\end{lemma}
Solving this differential inequality, we obtain 
$$
\| w(t) \|_{H^\sigma}^2 = \| u(t) -v(t) \|_{H^\sigma}^2 \le e^{Ct}  \| u _0-v_0 \|_{H^\sigma}^2. 
$$
Uniqueness of solutions follows immediately. 
\subsection{Continuous Dependence}
Here we shall show that the data-to-solution map $u_0 \mapsto u(t)$ for the ivp corresponding to the RZQ equation is continuous from $H^s$ to $C([0, T]; H^s)$. 

Fix $u_0 \in H^s$ and let $\{u_{0,n}\} \rightarrow u_0$ in $H^s$. Let $u$ be the solution to the RZQ equation with initial data $u_0$, and $u_n$ be the solution to the RZQ equation corresponding to initial data $u_{0,n}$. We shall show that 
$$
\lim_{n\rightarrow\infty} u_n = u \quad \text{ in } \quad C([0,T]; H^s). 
$$
It suffices to show that for $\eta >0$, there exists an $N>0$ such that for all $n>N$ we have $\|u-u_n\|_{C([0,T];H^s)}<\eta.$  To overcome the difficulty of the twisted Burgers terms in our estimates, we will use information from the mollified problem.  Let $u^\ee$ and $u_n^\ee$ represent the solutions to the mollified problem \eqref{mollifiedQR} with initial data $J_\ee u_0$ and $J_\ee u_{0,n}$ respectively.  Then, by an application of the triangle inequality, we have 
\begin{equation}\label{dep-est}
    \|u-u_n\|_{C([0,T];H^s)} \leq \|u-u^\ee\|_{C([0,T];H^s)} + \|u^\ee-u_n^\ee\|_{C([0,T];H^s)} + \|u_n^\ee-u_n\|_{C([0,T];H^s)}.
\end{equation}
Since we know that $u^\ee \to u$ and $u_n^\ee \to u_n$ in $C([0,T];H^s)$, we may choose $\ee>0$ sufficiently small to bound both the first and third terms in \eqref{dep-est} by $\eta/3$.  Therefore, it remains to find sufficient estimates on the difference $v=u^\ee-u_n^\ee$ in the $C([0,T];H^s)$ norm.

We first note that this difference satisfies the following equation
\begin{equation}\label{molli-dif}
    \begin{cases}
v_t+\frac12J_\ee\lambda^{-2}\p_x\left[\lambda^2 (J_\ee w)\lambda^2(J_\ee v)\right] = -\p_x\lambda^{-4}\left[\frac12(\lambda ^2\p_x w)(\lambda ^2\p_x v)+(\lambda^2 w)(\lambda^2 v)\right] \\ 
        v_0 = J_\ee u_0-J_\ee u_{0,n},
    \end{cases}
\end{equation}
where $w = u^\ee+u_n^\ee$.
Applying the operator $\lambda^s$ to \eqref{molli-dif}, multiplying by $\lambda^sv$ and integrating over the torus yields
\begin{align}
    \frac12\frac{d}{dt}\|v\|_{H^s}^2 &= -\int_{\ci}\lambda^s\left[\frac12J_\ee \lambda^{-2}\p_x\left[\lambda ^2(J_\ee w)\lambda^2(J_\ee v)\right]\right]\lambda^sv dx \nonumber \\ 
    &- \int_{\ci}\lambda^s\left[\p_x\lambda^{-4}\left[\frac12(\lambda ^2\p_x w)(\lambda^2 \p_x v)+(\lambda ^2w)(\lambda ^2v)\right]\right]\lambda^s v dx \nonumber \\
    &=I+II.
\end{align}
We first handle integral $II$ by applying Cauchy-Schwarz, the algebra property, and our solution size estimate to obtain
\begin{equation}\label{int2}
    |II| \lesssim \|w\|_{H^s}\|v\|_{H^s}^2 \lesssim \|v\|_{H^s}^2.
\end{equation}
Applying similar methods to the first integral along with the property $\|J_\ee f\|_{H^s} \leq \|f\|_{H^s}$, we find that
\begin{equation}
    |I| \lesssim \|J_\ee w\|_{H^{s+1}}\|J_\ee v\|_{H^{s+1}}\|v\|_{H^s}.
\end{equation}
We note that $\|f\|_{H^{s+1}} = \|f\|_{H^s}+\|\p_xf\|_{H^s}$ so that we may improve our estimate to 
\begin{equation}\label{int1}
    |I| \lesssim \frac{c}{\ee^2}\|v\|_{H^s}^2.
\end{equation}
Combining estimates \eqref{int2} and \eqref{int1}, we achieve the differential inequality
\begin{equation}\label{difeq1}
    \frac12\frac{d}{dt}\|v\|_{H^s}^2 \lesssim \frac{c}{\ee^2}\|v\|_{H^s}^2.
\end{equation}
Upon solving \eqref{difeq1}, we find obtain
\begin{equation}
    \|u^\ee - u_n^\ee\|_{H^s} \lesssim e^{\frac{c}{\ee^2}t}\|J_\ee(u_0 - u_{0,n})\|_{H^s} \leq e^{\frac{c}{\ee^2}t}\|u_0 - u_{0,n}\|_{H^s}.
\end{equation}
We may then choose $n$ sufficiently large so that 
\begin{equation}
    \|u_0 - u_{0,n}\|_{H^s} \leq e^{-\frac{c}{\ee^2}t}\frac{\eta}{3}.
\end{equation}
This gives us that the second term in \eqref{dep-est} may also be bounded by $\eta/3$ which completes the proof.

\textbf{Remark:}  Well-posedness on the real line may be proven in a similar manner.  Indeed, we may duplicate the proofs for continuous dependence and uniqueness.  In the subsection for existence, however, we must change our Friedrichs mollifier by fixing $j \in \mathcal{S}(\rr)$ with $\wh{j}(0) = 1$.  Then we may define $j_\ee(x) = (1/\ee)j(x/\ee)$ so that for each $\ee \in (0,1]$, we have $J_\ee f = j_\ee * f$.  Furthermore, in each of our replicated estimates we apply a cut-off function $\vp \in \mathcal{S}(\rr)$ so that we have convergence in the respective topologies as mentioned above.

  \section{Continuity of the data-to-solution map}
In this section, we shall prove that the data-to-solution map is not uniformly continuous. We will provide all of the details in the periodic case, and we will sketch the modifications and estimates needed in the non-periodic case. In order to prove these results, we find that it is convenient to write the RZQ equation in the following equivalent form
\begin{align}
      u_t  + uu_x + \frac12 \lambda^{-2} \partial_x (u_{xx})^2 
  & =  - \lambda^{-2}  \partial_x \left [(u_x)^2   \right ]  
  - \lambda^{-4} \partial_x \left [  ( \lambda^2   u)^2 +\frac12 (\lambda^2 u_x)^2 
   \right ]  .
\end{align}

\subsection{Nonuniform dependence on $\mathbb T$} 
Following the works cited in the introduction, we shall use the method of approximate solutions. We consider the following two sequences of functions corresponding to $\omega=\pm 1$
\begin{align}
    u^{ap}_{\omega, n}  (x,t) = \frac{ \omega }{n}+ \frac{1}{n^s}\cos(nx-\omega t). 
\end{align}
Our first result shows that these functions are indeed approximate solutions to the RZQ equation.
\begin{lemma}
We denote the error
\begin{align}
E \ \dot = &  \  \partial_t  u^{ap}_{\omega, n}  +   u^{ap}_{\omega, n} \partial_ x  u^{ap}_{\omega, n} + \frac12 \lambda^{-2} \partial_x ( \partial_x^2  u^{ap}_{\omega, n})^2 
 + \lambda^{-2}  \partial_x \left [( \partial_x  u^{ap}_{\omega, n} )^2   \right ]  
\\ &  +  \lambda^{-4} \partial_x \left [  ( \lambda  ^2   u^{ap}_{\omega, n} )^2 +\frac12 (\lambda ^2 \partial_x  u^{ap}_{\omega, n} )^2 
   \right ] ,
\end{align}
and have that for any $5/2< \sigma <s-1$, there exists a constant $C$ independent of $n$, such that
$$
\| E \|_{H^\sigma} \le C  \max\{ n^{\sigma+3-2s} , n^{2\sigma-2s}  \}  .
$$
\end{lemma}
\begin{proof}
We take the $H^\sigma$ norm of $E$, apply the triangle inequality, and use the following lemma which bounds the derivatives of $u^{ap}_{\omega, n}$. 
\begin{lemma}
For $s, \sigma>1$ and for all $k, j \in \mathbb N_0$, we have for a constant $C$ independent of $n$
$$
\|    \partial_t ^k \partial_x^j u^{ap}_{\omega, n}  \|_{H^\sigma} \le C  n^{s-\sigma - j} , \quad \|    \cos ^2 (nx-\omega t)  \|_{H^\sigma} \le C  n^{ \sigma}. 
$$
\end{lemma}
This lemma allows us to find  the following  bounds. 
\begin{align}
 & 
\left \|  \frac12 \lambda^{-2} \partial_x ( \partial_x^2  u^{ap}_{\omega, n})^2     \right \| _{H^\sigma} \le 
C n^{\sigma+3 -2s}
\\
& 
\left \|   \lambda^{-2}  \partial_x \left [( \partial_x  u^{ap}_{\omega, n} )^2   \right ]     \right \| _{H^\sigma} \le 
C n^{\sigma +1  -2s}
\\ &  
\left  \|   \lambda^{-4} \partial_x \left [  ( \lambda ^2    u^{ap}_{\omega, n} )^2 +\frac12 (\lambda^2  \partial_x  u^{ap}_{\omega, n} )^2 
   \right ]    \right  \| _{H^\sigma} \le 
C n^{2 \sigma -2  -2 s}+ 
C n^{2 \sigma  -2 s}. 
\end{align}
Also, we estimate the first two terms by
\begin{align} 
\left \| 
\partial_t  u^{ap}_{\omega, n}  +   u^{ap}_{\omega, n} \partial_ x  u^{ap}_{\omega, n}   \right  \| _{H^\sigma} 
=n^{1-2s} \left\| \sin(nx-\omega t) \cos(nx-\omega t) \right\|_{H^\sigma} 
\le C n^{\sigma +1-2s}. 
\end{align}
Combining the above estimates gives the desired bound for $E$. 
\end{proof}

We will also consider, for each $n=1,2,\dots$, the solution to the RZQ equation with initial data
$$
u_{0,\omega, n}(x) = u^{ap}_{\omega, n} (x,0).
$$
 Since the initial data is uniformly bounded for each $n = 1,2,\dots$, we have  a minimum lifespan, denoted $T>0$, for which all the solutions, $u_{\omega, n} $, exist for $0<t<T$. Furthermore, for $0<t<T$, from the solution size estimate we have   $\| u_{\omega, n}  (t)  \| _{H^\sigma} \le C n^{\sigma -s} $ for $ \sigma \in \mathbb R$. 
 
 Our next result shows that the actual solutions with the above initial data are approximated by the approximate solutions. We will denote the difference and sums by 
 \begin{align}
 v =  u^{ap}_{\omega, n}  -  u _{\omega, n} , \quad \text{ and } \quad  w=  u^{ap}_{\omega, n}   +  u _{\omega, n}.
 \end{align}
A straightforward calculation shows that $v$ satisfies the following ivp
\begin{align}\notag
 & \partial_t v = E -  \frac12 \partial_x\left (vw  \right)  -
 \frac12 \lambda^{-2} \partial_x \left  [\partial_x^2 v  \partial_x^2 w   \right  ]
 - \lambda^{-2}  \partial_x \left [\partial_x   v \partial_x w   \right ]  
 \\ & \ \ \ \ \  -  \lambda^{-4} \partial_x \left [    \lambda ^2     v \lambda ^2     w +\frac12    \lambda ^2   \partial_x   v \lambda  ^2 \partial_x     w
   \right ]
\\ & v(x,0) = 0 .
\end{align}
 Our next task is to show that the approximate solutions approximate the exact solutions in $H^s$. Our strategy is to estimate the difference in $H^\sigma $, $\sigma < s-1$ using the previous lemma, and then estimate the difference in $H^{s+1}$ and then interpolate. 
\begin{lemma}\label{error-periodic}
The $H^\sigma$, $5/2<\sigma<s-1$, norm of $v$ can be estimated by
$$
\| v\|_{H^\sigma} \le C  \max\{ n^{\sigma+3-2s} , n^{2\sigma-2s}  \}.
$$
\end{lemma}
\begin{proof}
We calculate the $H^\sigma$ energy of $v$ and we have 
\begin{align}
\frac12 \frac{d}{dt} \| v(t) \| _{H^\sigma} = &  \int \lambda^\sigma E \lambda^\sigma v dx - \frac12 \int \lambda^\sigma \partial_x (vw)  \lambda^\sigma v dx 
 - \frac12   \int \lambda^{\sigma-2} \partial_x ( \partial_x^2 v\partial_x^2 w)  \lambda^\sigma v dx  
\label{approx_sln}
\\
& -   \int \lambda^{\sigma-2} \partial_x ( \partial_x v\partial_x w)  \lambda^\sigma v dx  
-   \int \lambda^{\sigma-4} \partial_x (  \lambda^2  v\lambda^2 w   + \frac12 \lambda^2  \partial_x v\lambda^2  \partial_xw)   \lambda^\sigma v dx  .
 \notag
\end{align}
The first term is handled by the Cauchy-Schwarz inequality and our estimate on the error
$$
 \int \lambda^\sigma E \lambda^\sigma v dx \le  \| E \|_{H^\sigma }  \| v\|_{H^\sigma } \le    C  \max\{ n^{\sigma+3-2s} , n^{2\sigma-2s}  \}  \| v\|_{H^\sigma } .
$$ 
The second term is written as a commutator as follows
\begin{align}
  \int \lambda^\sigma \partial_x (vw)  \lambda^\sigma v dx 
  =   \int  \left( \lambda^\sigma \partial_x (vw)  -  w  \lambda^\sigma \partial_x   v\right)  \lambda^\sigma v dx  + 
  \int  w  \left( \lambda^\sigma \partial_x   v \right)  \lambda^\sigma v dx  .
\end{align}
We apply Lemma \ref{ccm} to obtain 
\begin{align}
  \int \lambda^\sigma \partial_x (vw)  \lambda^\sigma v dx 
   \le \| w\|_{H^s} \| v\|_{H^\sigma} ^2+ 
  \int  w  \left( \lambda^\sigma \partial_x   v \right)  \lambda^\sigma v dx  .
\end{align}
The remaining integral above is bounded using integration by parts followed by H\"older's inequality and the Sobolev embedding theorem
\begin{align}
  \int \lambda^\sigma \partial_x (vw)  \lambda^\sigma v dx 
   \le C \| v\|_{H^\sigma} ^2    .
\end{align}
Similarly, the remaining terms in equation \eqref{approx_sln} are bounded by $C \|v\|_{H^\sigma}^2$, the proofs are done in a similar way to the terms in the proof of uniqueness in the previous section, and we omit the details for the sake of brevity. 
Combining these estimates we obtain the following differential inequality for $ \|v\|_{H^\sigma}^2$
\begin{align}
\frac{d}{dt }  \|v\|_{H^\sigma}^2 \le C\left( \max\{ n^{\sigma+3-2s} , n^{2\sigma-2s}  \}  \| v\|_{H^\sigma } +  \| v\|_{H^\sigma }  ^2\right) .
\end{align}
Solving this yields
$$
 \|v(t) \|_{H^\sigma} \le e^{CT}  \max\{ n^{\sigma+3-2s} , n^{2\sigma-2s} \},
$$
which completes the proof of the lemma. 
\end{proof}

We now interpolate between $s_1 = \sigma$ and $s_2 = s+1$ to obtain 
$$
\| v\| _{H^s}
\le \|v\| _{H^\sigma} ^{\frac1{s+1-\sigma}}\|v\| _{H^{s+1}} ^{\frac{s-\sigma}{s+1-\sigma}}
\lesssim 
\left( \max\{ n^{\sigma+3-2s} , n^{2\sigma-2s} \right) ^{\frac1{s+1-\sigma}} n^{\frac{s-\sigma}{s+1-\sigma}} \lesssim n^{-\epsilon},
$$
where we used the interpolation inequality
 $$
 \|f\|_{H^s} \le \|f\|_{H^{s_1}} ^{(s_2-s)/(s_2-s_1) }  \|f\|_{H^{s_2}}^{(s-s_1)/(s_2-s_1) } ,
 $$
 and
where the exponent $\epsilon > 0$ is given by
$$
\epsilon = \min\left\{ \frac{s - \sigma}{s+1-\sigma} , \frac{s-3}{s+1-\sigma}\right\}.
$$
With this decay estimate, we are ready to prove the main theorem of this section.

We take the two sequences of solutions, $\{ u_{\omega, n}\}$ with uniform minimum lifespan $T>0$. We have that the initial data satisfies 
$$
\| u_{0,1, n}- u_{0,-1, n}\|_{H^s} \rightarrow 0, \text{ as }  n\rightarrow \infty. 
$$
At any later time, we have 
\begin{align}
\| u_{1,n} - u_{-1,n} \| _{H^s} 
 \ge &  \| u_{1,n}^{ap} - u_{-1,n}^{ap}  \| _{H^s} - \| u_{1,n}^{ap} - u_{1,n}  \| _{H^s} -  \| u_{-1,n}^{ap} - u_{-1,n}  \| _{H^s}
\\
& \ge \|  n^{-1} - n^{-s} \cos(nx)\sin(t)  \| _{H^s}  - n^{-\epsilon} \gtrsim |\sin(t)| - n^{-1} -n^{-\epsilon} . 
\end{align}
Applying $\liminf_{n\rightarrow \infty}$ on both sides yields 
\begin{align}
\liminf_{n\rightarrow\infty} \| u_{1,n} - u_{-1,n} \| _{H^s}  \gtrsim |\sin(t)| .
\end{align}
Hence, we have proved that there exists two sequences of solutions which remain bounded away from each other for all positive times, however, their initial data converge to each other; the data-to-solution map is not uniformly continuous.

\subsection{Nonuniform dependence on $\mathbb R$} The main differences in this subsection compared to the periodic case are: the construction of the approximate solutions and the estimates of the error. We provide details of these arguments, and refer the reader to the proof in the periodic case when the details are very similar.

We shall construct low-high frequency approximate solutions by first constructing low frequency {\em exact} solutions. We let $u_{\omega, n,  \ell}$ be the solution the the RZQ equation with the low frequency initial data $u_{0, \omega, n, \ell}  (x) = \omega n^{-1}    \varphi  (n^{-\delta }x) $ where $\varphi $ is a smooth bump function equal to $1$ when $|x|\le 2$. In other words, $u_{\omega, n,  \ell}$ is the solution to 
\begin{align}
& \partial_t u_{\omega, n,  \ell} + u_{\omega, n,  \ell} \partial_x u_{\omega, n,  \ell}  + \frac12 \lambda^{-2} \partial_x ( \partial_x ^2 u_{\omega, n,  \ell})^2 
  =   F(u_{\omega, n,  \ell}) 
   \\
    & u_{  \omega, n, \ell}  (x, 0)  = \omega n^{-1}   \varphi (n^{-\delta }x) .
\end{align}
where $\displaystyle F(u_{\omega, n,  \ell})= - \lambda^{-2}  \partial_x \left [(\partial_x u_{\omega, n,  \ell})^2   \right ]  
  - \lambda^{-4} \partial_x \left [  ( \lambda^2    u_{\omega, n,  \ell})^2 +\frac12 (\lambda^2  \partial_ x u_{\omega, n,  \ell})^2 
   \right ]  $. We will denote $u_\ell = u_{\omega, n,  \ell}$ whenever it is clear that the estimates we are computing are independent of $\omega$ and $n$. 
  From the well-posedness result, we have that for all $0<\delta$, the unique solution $u_{\omega, n,  \ell} \in C([0, 1]; H^s) $, $s>7/2$, satisfies, for all $\sigma\ge 0$ 
\begin{align}\label{l-est}
  \| u_{\omega, n,  \ell}(t) \|_{H^\sigma} \le c_\sigma n^{\frac{\delta}{2} - 1} . 
\end{align}

We then define the high frequency sequence of functions by
$$
u_{ h} ^{ap} = u_{\omega, n, h} ^{ap} (x,t) = n^{-\delta/2-s} \widetilde \varphi \left( \frac{x}{n^\delta} \right) \cos(nx-\omega t), 
$$
where $\widetilde \varphi (x)$ is a smooth bump function with support $|x|<  2$ and equal to $1$ when $|x|<1$, so that $\widetilde \varphi (x) \varphi (x)  =\widetilde \varphi (x) $. 
The high-frequency approximate solutions can be estimated in $H^\sigma$, for any $\sigma\ge 0$ using the following lemma found in \cite{himonas2009euler}. 
\begin{lemma}\label{hk2009}
Let $\varphi \in \mathcal S(\mathbb R)$, $\delta > 0$ and $\alpha \in \mathbb R$. Then for any $\sigma \ge 0 $ we have that 
$$
\lim_{n\rightarrow \infty} n^{-\sigma-\delta/2} \left \| \varphi \left( n^{-\delta} x\right) \cos(nx-\alpha) \right\| _{H^\sigma} = \frac{1}{\sqrt{2}} \| \varphi\|_{L^2} .
$$
The above relation also is true if $\cos$ is replaced by $\sin$.
\end{lemma}
So we see that the above high frequency approximate solutions are bounded in $H^s$. We now 
  define the sequence of low-high frequency approximate solutions by 
$$
u_{\omega, n} ^{ap} (x,t) =   u_{\omega, n,  \ell}  + u_{\omega, n, h} ^{ap}  = u_\ell+u_h^{ap} , 
$$
and we define the error by substituting this function into the RZQ equation. 
\begin{align}
E = \ &  \notag
\partial_t  u_h^{ap} + u_\ell \partial_x u_h^{ap}  +  u_h^{ap} \partial_x  u_h^{ap} + u_h^{ap} \partial_x u_\ell + 
 \frac12 \lambda^{-2} \partial_x ( \partial_x ^2 u_h^{ap} )^2 +  \lambda^{-2} \partial_x ( \partial_x ^2 u_h^{ap} \partial_x ^2 u_\ell ) 
\\ &   \notag
+ \lambda^{-2}  \partial_x \left [(\partial_x u_h ^{ap}  )^2   \right ]  + 2 \lambda^{-2}  \partial_x \left [(\partial_x u_h ^{ap} \partial_x u_\ell )    \right ]  
 + \lambda^{-4} \partial_x \left [  ( \lambda^2    u_h ^{ap}  )^2    \right ] 
 + 2 \lambda^{-4} \partial_x \left [  \lambda ^2   u_h ^{ap}   \lambda    u_\ell   \right ]  
\\ &
  +  \lambda^{-4} \partial_x  \left [  \frac12 (\lambda ^2 \partial_ x u_h^{ap}  )^2 
   \right ]  
     +  \lambda^{-4} \partial_x  \left [    \lambda ^2\partial_ x u_h^{ap}  \lambda ^2 \partial_ x u_\ell     \right ]  ,
\end{align}
where we used $u_\ell$ was an exact solution to the RZQ equation. 
We observe that the first two terms sum to 
\begin{align}\notag
\partial_t  u_h^{ap} + u_\ell \partial_x u_h^{ap} = &\ (u_\ell(x,0) - u_\ell(x,t) ) n^{-s-\delta/2 +1 } \widetilde \varphi(n^{-\delta } x) \sin(nx-\omega t) 
\\
& + u_\ell(x,t) n^{-s-3\delta/2}  \partial_x \widetilde \varphi (n^{-\delta }x) \cos(nx-\omega t) . 
\end{align}
We then decompose the error, $E$, into the following terms  
\begin{align*}
 &E_1 = (u_\ell(x,0) - u_\ell(x,t) ) n^{-s-\delta/2 +1 } \widetilde \varphi(n^{-\delta } x) \sin(nx-\omega t)   ,
\\
  &E_2 = u_\ell(x,t) n^{-s-3\delta/2}  \partial_x \widetilde \varphi (n^{-\delta }x) \cos(nx-\omega t)  ,
 \quad \quad E_3 =  u_h^{ap} \partial_x  u_h^{ap}   ,
  \\
 & 
 E_4 = u_h^{ap} \partial_x u_\ell  ,
\quad \quad 
E_5 =  \frac12 \lambda^{-2} \partial_x ( \partial_x ^2 u_h^{ap} )^2 ,
 \quad \quad  
 E_6 =  \lambda^{-2} \partial_x ( \partial_x ^2 u_h^{ap} \partial_x ^2 u_\ell ) ,
\\ &
E_7 = \lambda^{-2}  \partial_x \left [(\partial_x u_h ^{ap}  )^2   \right ] ,
\quad \quad  
E_8 = 2 \lambda^{-2}  \partial_x \left [(\partial_x u_h ^{ap} \partial_x u_\ell )    \right ]  ,
\quad \quad   E_9 =   \lambda^{-4} \partial_x \left [  ( \lambda ^2   u_h ^{ap}  )^2    \right ] ,
\\ & 
E_{10}=  2 \lambda^{-4} \partial_x \left [  \lambda  ^2  u_h ^{ap}   \lambda    u_\ell   \right ]  ,
\quad \quad E_{11} = \lambda^{-4} \partial_x  \left [  \frac12  \left (\lambda^2  \partial_ x u_h^{ap}   \right)^2  \right ]  ,
\quad \quad E_{12} = \lambda^{-4} \partial_x  \left [    \lambda ^2 \partial_ x u_h^{ap}  \lambda ^2\partial_ x u_\ell     \right ]  ,
\end{align*}
 so that $$
 E = E_1+E_2+\dots + E_{12} . 
 $$
 
 \subsubsection{Estimating the error.} For $s>7/2$, we estimate the error in $H^\sigma$ for some $s-2+\delta <\sigma < s-1$ as follows. 
 
 {\em Estimating $E_1$.} Using Lemma \ref{hk2009} we have 
 \begin{align}
 \| E_1 \|_{H^\sigma} \lesssim  n^{1+\sigma - s} \| u_\ell(x,0) - u_\ell(x,t)\|_{H^\sigma} .
 \end{align}
 using the fundamental theorem of calculus we obtain 
 \begin{align}
 \| E_1 \|_{H^\sigma} \lesssim  n^{1+\sigma - s} \int_0^t \| \partial_\tau u_\ell(x,\tau)  \|_{H^\sigma} d\tau.
 \end{align}
 Using the RZQ equation, the algebra property for Sobolev spaces, and estimate \eqref{l-est}, we have 
 \begin{align}
  \| \partial_\tau u_\ell(x,\tau)  \|_{H^\sigma} \le & \
\left \| u_{ \ell} \partial_x u_{  \ell}\right\|_{H^\sigma}  + \frac12\left \| \lambda^{-2} \partial_x ( \partial_x ^2 u_{  \ell})^2 \right\|_{H^\sigma}
+\left \|  F(u_{\ell}) \right\|_{H^\sigma}
 \lesssim  \left  \| u_{ \ell}\right\|_{H^{\sigma+1}} ^2 \lesssim n^{\delta-2} .
 \end{align}
 Hence we have the estimate 
  \begin{align}
 \| E_1 \|_{H^\sigma} \lesssim  n^{\delta+\sigma - s-1}  .
 \end{align}

 {\em Estimating $E_2$.}
 We have from the algebra property for Sobolev spaces 
   \begin{align}
 \| E_2 \|_{H^\sigma} \le  n^{-s-3\delta/2}  \| u_\ell(x,t) \|_{H^\sigma}  \|    \partial_x \widetilde \varphi (n^{-\delta }x) \cos(nx-\omega t)\|_{H^\sigma}   .
 \end{align}
 Applying estimate \eqref{l-est} and Lemma \ref{hk2009} we have
   \begin{align}
 \| E_2 \|_{H^\sigma} \lesssim  n^{-s-3\delta/2}   \cdot n^{ \delta/2-1} \cdot n^{\sigma +\delta/2}    = n^{\delta/2 + \sigma-s-1} .
 \end{align}

 {\em Estimating $E_3$.}
Since 
   \begin{align}
 \| E_3 \|_{H^\sigma}  
 =  \| u_h^{ap} \partial_x  u_h^{ap}  \|_{H^\sigma} 
 \lesssim 
  \| u_h^{ap}\|_{H^\sigma}  \|  \partial_x  u_h^{ap}  \|_{L^\infty}  +   \| \partial_x u_h^{ap}\|_{H^\sigma}  \|     u_h^{ap}  \|_{L^\infty}  
  .
 \end{align}
 Using $  \| u_h^{ap}\|_{H^\sigma} \lesssim n^{\sigma -s}$ and $  \|     u_h^{ap}  \|_{L^\infty}  \lesssim n^{-\delta/2-s}$ we have 
   \begin{align}
 \| E_3 \|_{H^\sigma}  
\lesssim  n^{\sigma -s} \cdot n^{1 -\delta/2-s} + n^{\sigma+1 -s} \cdot n^{-\delta/2-s}  = n^{\sigma+1 -2s-\delta/2},  \quad n\gg 1 
  .
 \end{align}

 {\em Estimating $E_4$.}
 We have 
    \begin{align*}
 \| E_4 \|_{H^\sigma}  
 = \ &  \|  u_h^{ap} \partial_x u_\ell    \|_{H^\sigma} 
 \lesssim 
  \| u_h^{ap}\|_{H^\sigma}  \|  \partial_x u_\ell    \|_{H^\sigma}   
  \lesssim n^{\sigma -s } \cdot n^{\delta/2-1} 
  ,
 \end{align*}
 which gives 
     \begin{align}
 \| E_4 \|_{H^\sigma}  
 \lesssim\ &  n^{\sigma +\delta/2  -s-1  } 
  , \quad n\gg 1 . 
 \end{align}

 {\em Estimating $E_5$.}
 We have 
    \begin{align*}
 \| E_5 \|_{H^\sigma}  
 = \ &  \|   \frac12 \lambda^{-2} \partial_x ( \partial_x ^2 u_h^{ap} )^2     \|_{H^\sigma} 
   \approx 
   \|    ( \partial_x ^2 u_h^{ap} )^2     \|_{H^{\sigma-1} } .
 \end{align*}
 A straightforward calculation shows
    \begin{align*}
  \partial_x ^2 u_h^{ap} = \ & -n^{2-s-\delta/2} \widetilde \varphi (n^{-\delta} x) \cos(nx-\omega t) -2 n^{1-s-3\delta/2}  \widetilde \varphi '(n^{-\delta} x) \sin(nx-\omega t)
  \\
  & +n^{-s-5\delta/2}\widetilde \varphi ''(n^{-\delta} x)\cos(nx-\omega t)
  .
 \end{align*}
Since $\sigma -1 >  1/2$, then $E_5$ can be bounded by 
     \begin{align*}
 \| E_5 \|_{H^\sigma}   \lesssim 
   \|    \partial_x ^2 u_h^{ap}       \|_{ L^\infty  }   \|    \partial_x ^2 u_h^{ap}       \|_{  H^{\sigma-1} }   \lesssim n^{2-s-\delta/2} \cdot n^{\sigma + 1-s }   = n^{\sigma + 3-2s-\delta/2}  \quad n \gg 1 .
 \end{align*}

 {\em Estimating $E_6$.}
  We have 
    \begin{align*}
 \| E_6 \|_{H^\sigma}  
 = \ &  \|   \lambda^{-2} \partial_x ( \partial_x ^2 u_h^{ap} \partial_x ^2 u_\ell )     \|_{H^\sigma} 
 \\ & \lesssim
   \|   \partial_x ^2 u_h^{ap}     \|_{H^{\sigma-1} }    \|    \partial_x ^2 u_\ell     \|_{ L^\infty}  +   \|   \partial_x ^2 u_\ell        \|_{H^{\sigma-1} }    \|    \partial_x ^2 u_h^{ap}     \|_{ L^\infty}  
    \\ & \lesssim
 n^{ \sigma+ 1 -s}  \cdot   n^{-2\delta -1} +   n^{-3\delta/2-1} \cdot n^{2-s-\delta/2} .
 \end{align*}
 From which we obtain 
     \begin{align*}
 \| E_6 \|_{H^\sigma}   \lesssim n^{\sigma-2\delta -s }  \quad n \gg 1  .
  \end{align*}
 
  {\em Estimating $E_7$.}
  We have 
    \begin{align*}
 \| E_7 \|_{H^\sigma}  
 = \ &  \|   \lambda^{-2}  \partial_x \left [(\partial_x u_h ^{ap}  )^2   \right ]     \|_{H^\sigma} 
 \\ & \lesssim
   \|   \partial_x   u_h^{ap}     \|_{H^{\sigma-1} }    \|    \partial_x    u_h^{ap}    \|_{ L^\infty} 
 \\ & \lesssim
 n^{ \sigma - s }  \cdot   n^{ 1-\delta/2-s} 
 \\ & 
   \approx  n^{ \sigma+1-\delta/2 - 2s } , \quad \text{ for } n \gg 1 .
 \end{align*}

  {\em Estimating $E_8$.}
  We have 
    \begin{align*}
 \| E_8 \|_{H^\sigma}  
 = \ &  \|   \lambda^{-1}  \partial_x \left [ \partial_x u_h ^{ap}   \partial_x u_\ell  \right ]     \|_{H^\sigma} 
 \\ & \lesssim
   \|   \partial_x   u_h^{ap}     \|_{H^{\sigma-1} }    \|    \partial_x    u_\ell    \|_{ L^\infty}  +    \|   \partial_x   u_\ell    \|_{H^{\sigma-1} }    \|    \partial_x    u_h^{ap}   \|_{ L^\infty}  
 \\ & \lesssim
 n^{ \sigma - s }  \cdot   n^{ \delta/2-1}  +  n^{  \delta/2-1 }  \cdot   n^{  1-\delta/2-s}
 \\ & 
   \approx  n^{ \sigma-1+\delta/2 -  s } , \quad \text{ for } n\gg 1 ,
 \end{align*}
 where we retained  the smallest power of $n$.

   {\em Estimating $E_9$.}
  We have 
    \begin{align*}
 \| E_9 \|_{H^\sigma}  
 = \ &  \|   \lambda^{-4}  \partial_x \left [ ( \lambda^2 u_h ^{ap}  )^2   \right ]     \|_{H^\sigma} 
   \lesssim 
  \|   ( \lambda^2 u_h ^{ap}  )^2       \|_{ H^{\sigma-3} } .
   \end{align*}
   If $\sigma-3 >0$, then this can be estimated as in the estimate for $E_5$. If $\sigma-3 < 0$, then we can proceed as follows
      \begin{align*}
 \| E_9 \|_{H^\sigma}  
  \lesssim &  \ 
  \|     \lambda^2 u_h ^{ap}         \|_{ L^\infty}   \|     \lambda ^2u_h ^{ap}         \|_{ L^2} 
  \lesssim 
 n^{-\delta/2-s+2} \cdot n^{2-s}   .
 \end{align*}
 From here we obtain 
     \begin{align*}
 \| E_9 \|_{H^\sigma}  \lesssim n^{4-2s-\delta/2}  , \quad \text{ for } n\gg 1.
 \end{align*}

   {\em Estimating $E_{10}$.}
As in the estimate for $E_9$, we show the details in the case $\sigma - 3<0$. The other case is less delicate. We have 
    \begin{align*}
 \| E_{10}\|_{H^\sigma}  
 = \ & 2  \|   \lambda^{-4}  \partial_x \left [  \lambda^2 u_h ^{ap}   \lambda u_\ell    \right ]     \|_{H^\sigma} 
 \\ &  \lesssim 
  \|  \lambda^2 u_h ^{ap}   \lambda^2 u_\ell        \|_{ L^2} 
 \\ &  \lesssim 
  \|     \lambda ^2u_h ^{ap}         \|_{ L^\infty}   \|     \lambda ^2u_\ell         \|_{ L^2} 
 \\ &  \lesssim 
 n^{-\delta/2-s+2} \cdot n^{ \delta/2-1} .
 \end{align*}
 From here we obtain 
     \begin{align*}
 \| E_{10} \|_{H^\sigma}  \lesssim n^{1-s }  , \quad n\gg 1 .
 \end{align*}

   {\em Estimating $E_{11}$.}
  We have 
    \begin{align*}
 \| E_{11}\|_{H^\sigma}  
 = \ & \frac12  \|   \lambda^{-4}  \partial_x \left [ \left( \lambda^2  \partial_x u_h ^{ap}  \right)^2    \right ]     \|_{H^\sigma} 
 \\ &  \lesssim 
  \|   \left( \lambda^2  \partial_x u_h ^{ap}  \right)^2        \|_{ H^{\sigma-3} }  .
  \end{align*}
  Using Lemma \ref{HH}, we with $r = \sigma-2$, we have 
    \begin{align*}
 \| E_{11}\|_{H^\sigma}  
 &  \lesssim 
  \|     \lambda ^2 \partial_x u_h ^{ap}        \|_{ H^{\sigma-2} }   \|     \lambda ^2 \partial_x u_h ^{ap}         \|_{ H^{\sigma-3} } 
 \\ &  \lesssim 
 n^{\sigma+1-s} \cdot n^{ \sigma-s} .
 \end{align*}
 From here we obtain 
     \begin{align*}
 \| E_{11} \|_{H^\sigma}  \lesssim n^{2\sigma-2s+1 }   , \quad \text{ for } n\gg 1.
 \end{align*}

   {\em Estimating $E_{12}$.}
As in the estimate for $E_9$, we show the details in the case $\sigma - 3<0$. The other case is less delicate. We have 
    \begin{align*}
 \| E_{12}\|_{H^\sigma}  
 = \ & \|   \lambda^{-4}  \partial_x \left [  \lambda ^2 \partial_x u_h ^{ap}    \lambda ^2 \partial_x u_\ell   \right ]     \|_{H^\sigma} 
  \\ &  \lesssim 
  \|   \lambda ^2 \partial_x u_h ^{ap}    \lambda ^2 \partial_x u_\ell       \|_{ L^2} 
 \\ &  \lesssim 
  \|     \lambda ^2\partial_x  u_h ^{ap}         \|_{ L^\infty}   \|     \lambda ^2\partial_x  u_\ell         \|_{ L^2} 
 \\ &  \lesssim 
 n^{-\delta/2-s+3} \cdot n^{ \delta/2-1} .
 \end{align*}
 From here we obtain 
     \begin{align*}
 \| E_{12} \|_{H^\sigma}  \lesssim n^{2-s-\delta/2 } \quad \text{ for } n\gg 1 .
 \end{align*}
 
 Combining the above twelve estimates, we have the following lemma. 
 \begin{lemma}
 Let $s>7/2$, $1/3<\delta<1$ and $2-\delta <\sigma < s-1$. Then we have the following estimate on the error
 \begin{align}
 \| E\|_{H^\sigma} \lesssim n^{-\alpha} , \quad \text{ for } n \gg 1,
 \end{align}
 where 
\begin{align} \label{alpha-def}
\alpha =  s+1-\sigma-\delta  >0.
\end{align}
 \end{lemma}
 
 \subsubsection{Estimating the distance between the actual and approximate solutions} 
 We consider two sequences of solutions to the RZQ equation by solving the initial value problem \ref{QR-nl}  with initial data 
 $$
 u_{\omega, n}(x,0) = u^{ap} _{\omega, n} (x,0 ) =  \omega n^{-1}   \varphi (n^{-\delta }x) 
 +n^{-\delta/2-s} \widetilde \varphi \left( \frac{x}{n^\delta} \right) \cos(nx ). 
 $$
 Since the initial data is bounded in $H^s$, by the well-posedness result, the collection of solutions exists on a time interval independent of $n$ and $\omega$. Additionally, it is easy to check that the difference $   u_{1, n}(x,0)  -  u_{-1, n}(x,0) $ goes to zero in $H^s$, and the difference $u^{ap} _{1, n} (x,t )-u^{ap} _{-1, n} (x,t )$ remains bounded away from each other in $H^s$ for $t>0$. 
 
 To prove that the actual solutions remain bounded away from each other for any positive time, we will show that the difference between the actual and approximate solutions decay as $n\rightarrow \infty$. For this we let 
 \begin{align}
 v =  u^{ap}_{\omega, n}  -  u _{\omega, n} , \quad \text{ and } \quad  w =  u^{ap}_{\omega, n}   +  u _{\omega, n}.
 \end{align}
A straightforward calculation shows that $v$ satisfies the following ivp
\begin{align}\notag
 & \partial_t v = E -  \frac12 \partial_x\left (vw  \right)  -
 \frac12 \lambda^{-2} \partial_x \left  [\partial_x^2 v  \partial_x^2 w   \right  ]
 - \lambda^{-2}  \partial_x \left [\partial_x   v \partial_x w   \right ]  
 \\ & \ \ \ \ \  -  \lambda^{-4} \partial_x \left [    \lambda  ^2    v \lambda   ^2   w +\frac12    \lambda  ^2  \partial_x   v \lambda   ^2\partial_x     w
   \right ]
\\ & v(x,0) = 0 .
\end{align}
The following lemma can be proved in a methodology similar to Lemma \ref{error-periodic} in the periodic case. 
 \begin{lemma} \label{error-nonperiodic}
The $H^\sigma$, $\max\{5/2, 2-\delta\} <\sigma<s-1$, norm of $v$ can be estimated by
$$
\| v\|_{H^\sigma} \le C n^{-\alpha} .
$$
where $\alpha = s+1-\sigma-\delta  >0$ and $1/2<\delta<1$ (see \eqref{alpha-def}). 
\end{lemma}

 \subsubsection{Completing the proof on $\mathbb R$} 
 Let $s > 7/2$ be fixed. We consider $u_{1, n}(x,t) $ and $u_{-1, n}(x,t) $, the unique solutions to the ivp for the RZQ equation with initial data $u^{ap} _{1, n} (x,0) $ and  $u^{ap} _{-1, n} (x,0) $ respectively. From the well-posedness argument, these solutions belong to $C([0, T]; H^s)$ for some  $T>0$ independent of $n\gg1 $ and $1/3<\delta <1$. 
 Additionally, applying the well-posedness theorem in $H^k$ for $k = \lfloor s+2 \rfloor$, we have 
 $$
 \| u_{\omega, n} (t) \|_{H^k } \lesssim  \| u_{\omega, n} ^{ap} (0) \|_{H^k } , \quad 0\le t \le T , \quad n\gg1.  
 $$
 A simple calculation shows 
  $$
  \| u_{\omega, n} ^{ap} (0) \|_{H^k }  \lesssim  n^{k-s} .  
 $$
 Therefore, we obtain the following estimate for the difference of the actual and approximate solutions in $H^k$ for $0\le t \le T$ 
 \begin{align}
  \| u_{\omega, n} (t) - u_{\omega, n}^{ap}  (t)  \|_{H^k } \lesssim n^{k-s} , \quad n\gg1 . 
 \end{align}
 Applying the interpolation inequality
 $$
 \|f\|_{H^s} \le \|f\|_{H^{s_1}} ^{(s_2-s)/(s_2-s_1) }  \|f\|_{H^{s_2}}^{(s-s_1)/(s_2-s_1) } ,
 $$
 with $s_2 = k$ and $s_1 = \sigma$ we obtain 
  \begin{align}
  \| u_{\omega, n} (t) - u_{\omega, n}^{ap}  (t)  \|_{H^s } \le  & \   \| u_{\omega, n} (t) - u_{\omega, n}^{ap}  (t)  \|_{H^\sigma }   ^{(k-s)/( k-\sigma) }   \| u_{\omega, n} (t) - u_{\omega, n}^{ap}  (t)  \|_{H^k }^{(s-\sigma)/( k-\sigma) }   
  \notag \\
  & \lesssim  n^{-\alpha(k-s)/( k-\sigma) } n^{(k-s) (s-\sigma )/( k-\sigma) }
  \notag  \\
   & \lesssim    n^{-(\sigma -s+ \alpha)(k-s)/( k-\sigma) } 
 \notag   \\
   & \lesssim    n^{-(1-\delta)(k-s)/( k-\sigma) } , \quad n\gg1.
 \end{align}
 It is easy to see the above exponent is negative, and the difference between the actual and approximate solutions tends to zero as $n\rightarrow \infty$. 
 From here, a calculation which mirrors the periodic case shows that the approximate solutions remain bounded away from each for all positive time, and in particular 
\begin{align}
\liminf_{n\rightarrow\infty} \| u_{1,n} - u_{-1,n} \| _{H^s}  \gtrsim |\sin(t)| .
\end{align}
This complete the proof in the non-periodic case.

\section{Proof of Theorem \ref{illp}} 

We first note that the pseudo-peakon solutions of the RZQ equation found in \cite{qr} are of the form
\[
u = \frac{c}{2}e^{-|\xi|}(1+|\xi|), \ \ \ \xi=x-ct.
\]
We seek to show that these solutions live in the Sobolev space $H^s$ for $s<7/2$.  Let's consider the pseudo-peakon centered at $t=0$ and  with $c=1$.  By taking the Fourier transform, we find that
\[
\wh{u} = \int_\rr e^{-ix\xi}u(x)dx = \frac{1}{1+\xi^2}+\int_{0}^\infty xe^{-x}\cos(x\xi)dx,
\]
where the last integral was obtained by symmetry of our function and the fact that $e^{-ix\xi}=\cos(x\xi)-i\sin(x\xi)$.  To compute this remaining integral, consider
\[
F(\xi) = \int_0^\infty e^{-x}\sin(x\xi)dx=\frac{\xi}{1+\xi^2},
\]
which is obtained via integration by parts.  Then we have that
\[
F'(\xi) = \frac{d}{d\xi}\int_0^\infty e^{-x}\sin(x\xi)dx=\int_0^\infty xe^{-x}\cos(x\xi)dx = \frac{1-\xi^2}{(1+\xi^2)^2}.
\]
Thus, we find that
\[
\wh{u} = \frac{2}{(1+\xi^2)^2}.
\]
Now we consider the Sobolev norm of our pseudo-peakons.  We see that
\[
\|u\|_{H^s}^2 = 4\int_\rr(1+\xi^2)^{s-4}d\xi,
\]
which is bounded for $s<7/2$.  Thus, $u \in H^s$ for $s<7/2$.  

From \cite{rzq} and again in \cite{qr}, we note that the multi-pseudo-peakon solutions are of the form
\[
u=\sum_{j=1}^N\frac{p_j(t)}{2}e^{-|x-q_j(t)|}(1+|x-q_j(t)|),
\]
where $p_j(t), q_j(t)$ satisfy the following canonical Hamiltonian dynamical system
\[
\dot{p_j} = -\frac{\p H}{\p q_j}, \ \ \dot{q_j} = \frac{\p H}{\p p_j},
\]
with the Hamiltonian function
\[
H=\frac12\sum_{i,j=1}^N p_ip_je^{-|q_i-q_j|}.
\]
Since this coincides exactly with the finite-dimensional peakon dynamical system of the CH equation, we may apply the same methods from \cite{hhg} to find that the RZQ equation is ill-posed in $H^s$ for $s<7/2$.

\section{Proof of H\"older Continuity}
Let $u,w$ be two solutions to the RZQ equation with initial data $u_0,w_0,$ respectively. 
Define $v = u-w, v_0 = u_0 - w_0.$ 
Then  
    \[v_t =  -\frac12\lambda^{-2}\partial_x\left[\lambda^2v\lambda^2(u+w)\right] - \frac12\lambda^{-4}\partial_x\left[2\lambda^2v\lambda^2(u+w) + \lambda^2v_x\lambda^2(u+w)_x\right].\] 
The $H^r$ energy of $v$ is then given by
    \[\frac12 \frac{d}{dt} \|v(t)\|_{H^r}^2 = 
        - \frac12\int \left[\lambda^{r-2}\partial_x\left[\lambda^2v\lambda^2(u+w)\right] + \lambda^4\partial_x\left[2\lambda^2v\lambda^2(u+w) + \lambda^2v_x\lambda^2(u+w)_x\right]\right]\lambda^{r}vdx.\]
\noindent Taking the absolute value and applying the triangle inequality we have 
 \begin{align} 
  \notag  \left|\frac12 \frac{d}{dt} \|v(t)\|_{H^r}^2 \right| 
        &\lesssim \left| \int \lambda^{r-2}\partial_x\left[\lambda^2v\lambda^2(u+w)\right] \lambda^{r}vdx\right|
        \\ &+ \left| \int \lambda^{r-4}\partial_x\left[2\lambda^2v\lambda^2(u+w) + \lambda^2v_x\lambda^2(u+w)_x\right] \lambda^{r}vdx\right|. 
         \label{energy}
\end{align}
 
\noindent \textbf{Lipschitz in $A_1$}. To bound the first term of \eqref{energy}, we note  
\begin{align}
\notag    \left| \int \lambda^{r-2}\partial_x\left[\lambda^2v\lambda^2(u+w)\right] \lambda^{r}vdx\right|
        &\leq \left| \int \lambda^{r-2}\left[-vz_{xxx} + v_{xx}z_{xxx} + v_{xxx}(z_{xx}-z)\right]\lambda^{r}vdx\right|
            \\&+ \left| \int \lambda^{r-2}\partial_x\left[vz - v_xz_x\right]\lambda^{r}vdx\right|,\label{d2energy}
\end{align}
where $z = u + w.$
Rewriting the first term of \eqref{d2energy} and using the triangle inequality, we have
    \begin{align}
         \notag \left| \int \lambda^{r-2}\left[-vz_{xxx} + v_{xx}z_{xxx} + v_{xxx}(z_{xx}-z)\right]\lambda^{r}vdx\right|
         & \leq \left|\int \lambda^{r-2}\left[-vz_{xxx} + z_xv_{xx}\right]\lambda^{r}vdx\right| 
            \\ & + \left| \int \lambda^{r-2} \partial_x((z_{xx}-z)v_{xx})\lambda^{r}vdx\right|
            \label{split}
    \end{align} 
We may handle the new second term by commuting $\lambda^{r-2}\partial_x$ with $(z_{xx}-z)$.
This replacement along with the triangle inequality yields   
    \begin{align}
    \notag    \left| \int \lambda^{r-2}\left[-vz_{xxx} + v_{xx}z_{xxx} + v_{xxx}(z_{xx}-z)\right]\lambda^{r}vdx\right|
            &\leq \left|\int \lambda^{r-2}\left[-vz_{xxx} + z_xv_{xx}\right]\lambda^{r}vdx\right|   
            \\ 
            \notag &+ \left| \int \left[\lambda^{r-2} \partial_x, (z_{xx}-z)\right]v_{xx}\lambda^{r}vdx \right|
            \\ &+ \left| \int (z_{xx}-z)\lambda^{r-2} v_{xxx}\lambda^{r}vdx \right|.\label{comm}
    \end{align} 
To bound the first term of \eqref{comm}, we apply the Cauchy-Schwarz inequality to get
\begin{equation}\label{cs2}
    \left|\int \lambda^{r-2}\left[-vz_{xxx} + z_xv_{xx}\right]\lambda^{r}vdx\right|  
         \leq \|vz_{xxx}\|_{H^{r-2}}\|v\|_{H^r} 
            + \|z_xv_{xx}\|_{H^{r-2}}\|v\|_{H^r}.
\end{equation}
To proceed, we require the following   estimate. 
\begin{lemma}\label{lemma3}
If $1 \leq \mu \leq 3, \sigma - 3 > \frac12,$ and $\mu + \sigma \geq 6,$ then 
    \[\|fg\|_{H^{\mu-3}} \leq c_{\mu,\sigma} \|f\|_{H^{\sigma-3}} \|g\|_{H^{\mu-3}}.\]
\end{lemma}
\begin{proof}
    We provide the details of Lemma \ref{lemma3} on $\mathbb{R}$, the periodic case is similar. 
We first note 
\begin{equation}
    \| fg \|_{H^{\mu - 3}}^2 = \int(1 + \xi^2)^{\mu-3} \left|\widehat{fg}(\xi)\right|^2 d\xi 
\end{equation}
Using the definition of convolution,
\begin{equation*}
    \| fg \|_{H^{\mu - 3}}^2 = \int(1 + \xi^2)^{\mu-3} \left|\int \hat{f}(\eta) \hat{g}(\xi - \eta)d\eta\right|^2 d\xi, 
\end{equation*}
which may be rewritten as 
\begin{equation}
    \| fg \|_{H^{\mu - 3}}^2 = 
        \int(1 + \xi^2)^{\mu-3} \left|\int (1+\eta^2)^{\frac{\sigma-3}{2}}\hat{f}(\eta)(1+\eta^2)^{\frac{3-\sigma}{2}}\hat{g}(\xi - \eta)d\eta\right|^2 d\xi. 
\end{equation}
Then by Cauchy-Schwarz, 
\begin{equation}
    \| fg \|_{H^{\mu - 3}}^2 \leq 
        \|f\|_{H^{\sigma - 3}} \int(1 + \xi^2)^{\mu-3}\int(1+\eta^2)^{3-\sigma}\left| \hat{g}(\xi - \eta)\right|^2 d\eta d\xi. 
\end{equation}
Implementing the change of variables $\tilde{\eta} = \xi - \eta$ and dropping the tildes, we have 
\begin{equation}
    \| fg \|_{H^{\mu - 3}}^2 \leq 
        \|f\|_{H^{\sigma - 3}} \int(1 + \xi^2)^{\mu-3}\int(1+(\xi - \eta)^2)^{3-\sigma}\left| \hat{g}(\eta)\right|^2 d\eta d\xi. 
\end{equation}
A change of the order of integration then yields
\begin{equation}\label{above}
    \| fg \|_{H^{\mu - 3}}^2 \leq 
        \|f\|_{H^{\sigma - 3}} \int \left| \hat{g}(\eta)\right|^2 \int(1 + \xi^2)^{\mu-3}(1+(\xi - \eta)^2)^{3-\sigma} d\xi d\eta. 
\end{equation}
For the integral in \eqref{above} to be bounded by $c_{\mu,\sigma}\|g\|_{\mu-3},$ it remains to show
\begin{equation}\label{wwts}
    I(\eta) := \int(1 + \xi^2)^{\mu-3}(1+(\xi - \eta)^2)^{3-\sigma}                 d\xi 
            \leq c_{\mu,\sigma}(1+\eta^2)^{\mu-3}.
\end{equation}
After making the change of variables $\tilde{\xi} = \xi - \eta$, dropping the tilde, and rearranging, we arrive at 
\begin{equation}
    I(\eta) = \int \frac{d \xi}{(1+(\xi - \eta)^2)^{3-\mu}(1 + \xi^2)^{\sigma-3}}.
\end{equation}
Here, we note that since $3 - \mu > 0$ and $\sigma - 3 > \frac12,$ the integral is finite for all $\eta.$ 
Additionally, we may consider only the case in which $\eta \geq 0$ since when $\eta < 0,$ the change of variables $\tilde{\xi} = -\xi$ returns the problem to the $\eta \geq 0$ case.  
Lastly, for $\eta >0 $ and $\xi \leq 0$, we have $\xi^2 \geq 2\xi\eta$ which gives $(1+(\xi - \eta)^2)^{3-\mu}\geq(1 + \eta^2)^{3-\mu}$. 
Thus \eqref{wwts} holds on $(-\infty, 0]$ and we are left to show 
\begin{equation}
    J(\eta) := \int_0^{\infty}\frac{d \xi}{(1+(\xi - \eta)^2)^{3-\mu}(1 + \xi^2)^{3-\sigma}} \leq \frac{c_{\mu,\sigma}}{(1 + \eta^2)^{3-\mu}}, \ \ \eta>0.
\end{equation}
We will estimate by breaking the domain of integration into the subintervals 
    $[0, \frac\eta2], [\frac\eta2,\eta], [\eta,\frac{3\eta}{2}],$ $ [\frac{3\eta}{2},\infty)$. 
For $\xi \in [0,\frac\eta2 ],$ we have 
    $1 + (\xi - \eta)^2 \geq 1 + \frac{\eta^2}{4}$ and thus 
\begin{equation*}
    \int_0^{\frac{\eta}{2}}\frac{d \xi}{(1+(\xi - \eta)^2)^{3-\mu}(1 + \xi^2)^{\sigma-3}} 
        \leq \frac{1}{(1+\eta^2/4)^{3-\mu}}\int_0^{\frac{\eta}{2}}\frac{d \xi}{(1 + \xi^2)^{\sigma-3}}.
\end{equation*}
Since $\mu - 3 \leq 0$, we see 
    $(1 + \frac{\eta^2}{4})^{\mu-3} = \left[\frac14( 4 + \eta^2)\right]^{\mu-3} \leq 4^{3-\mu}(1 + \eta^2)^{\mu-3}$.
Therefore 
\begin{equation}
    \int_0^{\frac{\eta}{2}}\frac{d \xi}{(1+(\xi - \eta)^2)^{3-\mu}(1 + \xi^2)^{\sigma-3}} 
        \leq \frac{4^{3-\mu}c_s}{(1+\eta^2)^{3-\mu}},
\end{equation}
where 
    \[c_s = \int_0^{\infty}\frac{d \xi}{(1 + \xi^2)^{\sigma-3}}.\]
The interval $\xi \in \left[\frac{3\eta}{2}, \infty
\right)$ can be estimated in the same manner since the inequality $1 + (\xi - \eta)^2 \geq 1 + \frac{\eta^2}{4}$ still holds. 
As such, 
\begin{equation}
    \int_{\frac{3\eta}{2}}^{\infty}\frac{d \xi}{(1+(\xi - \eta)^2)^{3-\mu}(1 + \xi^2)^{\sigma-3}} 
        \leq \frac{4^{3-\mu}c_s}{(1+\eta^2)^{3-\mu}},
\end{equation}
For $\xi \in \left[\frac\eta2,\eta\right],$ we have $(1 + \xi^2)^{\sigma - 3} \geq (1 + \frac{\eta^2}{4})^{\sigma-3}$. 
We additionally note that for $a \geq 0, 1 + a^2 \geq \frac12(1+a)^2.$
Letting $a = \eta - \xi$ and raising to the $3 - \mu$ power yields
    $(1 + (\xi - \eta)^2)^{3-\mu} \geq 2^{\mu - 3}(1 + \eta - \xi)^{2}$. 
Then 
\begin{equation}\label{J3}
    \int_{\frac{\eta}{2}}^{\eta}\frac{d \xi}{(1+(\xi - \eta)^2)^{3-\mu}(1 + \xi^2)^{\sigma-3}} 
        \leq \frac{2^{3-\mu}4^{\sigma-3}}{(1+\eta^2)^{\sigma-3}}\int_{\frac{\eta}{2}}^{\eta}\frac{d\xi}{(1+\eta-\xi)^{2(3-\mu)}}.
\end{equation}
It remains to show 
\begin{equation*}
    K(\eta) := 
         \int_{\frac{\eta}{2}}^{\eta}\frac{d\xi}{(1+\eta-\xi)^{2(3-\mu)}} < \infty.
\end{equation*}
We consider three cases: $\mu \in \left[1,\frac52\right), \mu = \frac52,$ and $\mu \in \left(\frac52, 3\right]$. 
For $\mu \in \left[1,\frac52\right)$, we have $2\mu-5 < 0$ and so evaluating the integral gives 
\begin{equation}\label{K1}
    K(\eta) = 
         \frac{1}{5-2\mu}\left[1-(1+\eta/2)^{2\mu-5}\right] 
         \leq \frac{1}{5-2\mu}.
\end{equation}
Combining \eqref{J3} and \eqref{K1} and using the fact that $3-\sigma \leq \mu - 3$, we have
\begin{equation}
    \int_{\frac{\eta}{2}}^{\eta}\frac{d \xi}{(1+(\xi - \eta)^2)^{3-\mu}(1 + \xi^2)^{\sigma-3}} 
        \leq \frac{2^{3-\mu}4^{\sigma-3}}{(5-2\mu)(1+\eta^2)^{\sigma-3}}
        \leq \frac{2^{3-\mu}4^{\sigma-3}}{(5-2\mu)(1+\eta^2)^{3-\mu}}.
\end{equation}
For $\mu \in \left(\frac52,3\right]$, we have $2\mu-5 > 0$ and thus
\begin{equation*}
    K(\eta) = 
         \frac{1}{2\mu-5}\left[(1+\eta/2)^{2\mu-5})-1\right] 
         \leq \frac{1}{2\mu-5}\left(1+\eta/2\right)^{2\mu-5}.
\end{equation*}
Noting that for $a \geq 0, \frac12(1+a)^2 \leq 4 + a^2$ and substituting in $a = \frac\eta2,$ we find 
    $(1 + \frac\eta2)^{2\mu-5} \leq 2^{\mu-5/2}(1 + \eta^2)^{\mu-5/2}.$ 
The above estimate then becomes
\begin{equation}\label{K2}
    K(\eta) \leq  
         \frac{2^{\mu-5/2}(1 + \eta^2)^{\mu-5/2}}{2\mu-5}.
\end{equation}
Combining \eqref{J3} and \eqref{K2} and using the fact that $(\sigma-3) - (\mu - 5/2) \geq 3 - \mu$, we have
\begin{equation}
    \int_{\frac{\eta}{2}}^{\eta}\frac{d \xi}{(1+(\xi - \eta)^2)^{3-\mu}(1 + \xi^2)^{\sigma-3}} 
        \leq \frac{2^{1/2}4^{\sigma-3}}{(2\mu-5)(1+\eta^2)^{3-\mu}}.
\end{equation}
Finally, if $\mu = \frac52,$ 
\begin{equation}\label{K3}
    K(\eta) = 
        \int_{\frac{\eta}{2}}^{\eta}\frac{d\xi}{(1+\eta-\xi)}
        = \ln\left(1 + \frac\eta2\right).
\end{equation}
Since $\ln\left(1 + \frac\eta2\right) \leq c_{\sigma}'\left(1 + \frac\eta2\right)^{\sigma-7/2}$, we have by combining \eqref{J3} and \eqref{K3}, 
\begin{equation}
    \int_{\frac{\eta}{2}}^{\eta}\frac{d \xi}{(1+(\xi - \eta)^2)^{3-\mu}(1 + \xi^2)^{\sigma-3}} 
        \leq \frac{2^{3-\mu}4^{\sigma-3}c_{\sigma}}{(1+\eta^2)^{3-\mu}}.
\end{equation}
Therefore for $\mu \in \left[1,3\right],$
\begin{equation}
    \int_{\frac{\eta}{2}}^{\eta}\frac{d \xi}{(1+(\xi - \eta)^2)^{3-\mu}(1 + \xi^2)^{\sigma-3}} 
        \leq \frac{c_{\mu,\sigma}}{(1+\eta^2)^{3-\mu}}.
\end{equation}
For $\xi \in \left[\eta,\frac{3\eta}{2}\right]$, we note that $(1+\xi^2)^{\sigma - 3} \geq (1 + \eta^2)^{\sigma - 3}$. 
Using the same identity as in \eqref{J3} now with $a = \xi - \eta$, we have $(1 + (\xi - \eta)^2) \geq 2^{\mu-3}(1 + \xi - \eta)^2$. 
Thus
\begin{equation}\label{J4}
    \int_{\eta}^{\frac{3\eta}{2}}\frac{d \xi}{(1+(\xi - \eta)^2)^{3-\mu}(1 + \xi^2)^{\sigma-3}} 
        \leq \frac{2^{3-\mu}}{(1+\eta^2)^{\sigma-3}}\int^{\frac{3\eta}{2}}_{\eta}\frac{d\xi}{(1+\xi-\eta)^{2(3-\mu)}}.
\end{equation}
Remaining calculations follow similarly to the above case, and we again obtain the bound 
\begin{equation}
    \int_{\eta}^{\frac{3\eta}{2}}\frac{d \xi}{(1+(\xi - \eta)^2)^{3-\mu}(1 + \xi^2)^{\sigma-3}} 
        \leq \frac{c_{\mu,\sigma}}{(1+\eta^2)^{3-\mu}}.
\end{equation}
Thus the proof of Lemma \ref{lemma3} on $\mathbb{R}$ is complete.
\end{proof}
\noindent Thus if $r - 1 > \frac12$, then by Lemma \ref{HH},
    \begin{align*}
        \left|\int \lambda^{r-2}\left[-vz_{xxx} + z_xv_{xx}\right]\lambda^{r}vdx\right|
        &\leq c_{r}\|v\|_{H^{r-1}}\|z_{xxx}\|_{H^{r-2}}\|v\|_{H^r} 
            \\&+ c_{r}\|z_x\|_{H^{r-1}}\|v_{xx}\|_{H^{r-2}}\|v\|_{H^r}.
    \end{align*}

Replacing $z = u + w$ and the solution size estimate, we have for $r+1 < s$
    \[\left|\int \lambda^{r-2}\left[-vz_{xxx} + z_xv_{xx}\right]\lambda^{r}vdx\right|   
        \leq   c_{r,\rho}\|v\|_{H^r}^2.\]
If instead $1 < r + 1 < 3$, $s - 3 > \frac12,$ and $r + 1 + s \geq 6$, Lemma \ref{lemma3} gives
\begin{equation*}
        \left|\int \lambda^{r-2}\left[-vz_{xxx} + z_xv_{xx}\right]\lambda^{r}vdx\right|
        \leq c_{r,s}\|v\|_{H^{s-3}}\|z_{xxx}\|_{H^{r-2}}\|v\|_{H^r} 
                + c_{r,s}\|z_x\|_{H^{s-3}}\|v_{xx}\|_{H^{r-2}}\|v\|_{H^r},
\end{equation*}
and so for $r + 1 < s,$
\begin{equation}\label{term1pt1est}
    \left|\int \lambda^{r-2}\left[-vz_{xxx} + z_xv_{xx}\right]\lambda^{r}vdx\right| 
        \leq c_{r,s}(\|u_{0}\|_{H^{s}} + \|w_0\|_{H^s})\|v\|^2_{H^r} \leq c_{r,s,\rho}\|v\|^2_{H^r}.
\end{equation}
To bound the second term of inequality \eqref{comm}, we require the  Calderon-Coifman-Meyer type commutator estimate, Lemma \ref{ccm}, and the Cauchy-Schwarz inequality.
 Applying these estimates we arrive at
\begin{align} \notag
    \left| \int \left[\lambda^{r-2} \partial_x, (z_{xx}-z)\right]v_{xx}\lambda^{r}vdx \right| 
        &\leq \left\|\left[\lambda^{r-2} \partial_x, (z_{xx}-z)\right] v_{xx} \right\|_{L^2} \|v\|_{H^r} 
        \\ &
            \leq c_{r,s}\|(z_{xx}-z)\|_{H^s}\|v\|^2_{H^r}\label{term1pt2est}
,
\end{align}
for $1 \leq r \leq s-1$ and $s > \frac52$.
For the third term of inequality \eqref{comm}, integrating by parts yields
    \begin{align}\label{ibp}
        \left| \int (z_{xx}-z)\lambda^{r-2} v_{xxx}\lambda^{r}vdx \right|  & \leq \left| \int (z_{xx}-z)_x\lambda^{r-2} v_{xx}\lambda^{r}vdx \right|  \notag
        \\ & + \left| \int (z_{xx}-z)\lambda^{r-2} v_{xx}\lambda^{r}v_xdx \right|.
    \end{align} 
To handle the first term on the right-hand side of inequality \eqref{ibp}, we note by Cauchy-Schwarz and the solution-size estimate
    \[\left| \int (z_{xx}-z)\lambda^{r-2} v_{xx}\lambda^{r}vdx \right| 
        \leq \|(z_{xx}-z)_x\|_{L^{\infty}} \|v_{xx}\|_{H^{r-2}}\|v\|_{H^r} 
       \leq c_{\rho}\|v\|^2_{H^r}.\]
To handle the right-hand side of \eqref{ibp}, we notice
    \begin{equation}\label{add0}
        \left| \int (z_{xx}-z)\lambda^{r-2} v_{xx}\lambda^{r}v_xdx \right| = \left| \int (z_{xx}-z)\lambda^{r}v \lambda^{r}v_xdx + (z_{xx}-z)\lambda^{r-2} v\lambda^{r}v_xdx \right|.
    \end{equation}
Additional integration by parts and applications of Cauchy-Schwarz then yields
\begin{align*}
    \left| \int (z_{xx}-z)\lambda^{r-2} v_{xx}\lambda^{r}v_xdx \right| 
        &\leq \left| \frac12\int (z_{xx}-z)_x (\lambda^{r}v)^2dx \right| + \left| \int (z_{xx}-z)\lambda^{r-2} v_x\lambda^{r}vdx \right| 
        \\ & + \left| \int (z_{xx}-z)_x\lambda^{r-2} v\lambda^{r}vdx \right| 
            \\ &\leq \|(z_{xx}-z)_x\|_{L^{\infty}}\|v\|^2_{H^r} + \|(z_{xx}-z)\|_{L^{\infty}}\|v_{x}\|_{H^{r-2}}\|v\|_{H^r} 
            \\ & + \|(z_{xx}-z)_x\|_{L^{\infty}}\|v\|_{H^{r-2}}\|v\|_{H^r},
\end{align*}
which by the solution-size estimate  becomes
\begin{equation}\label{term1pt3est}
    \left| \int (z_{xx}-z)\lambda^{r-2} v_{xx}\lambda^{r}v_xdx \right| \leq c_{\rho}\|v\|^2_{H^r}.
\end{equation}
Combining estimates \eqref{term1pt1est}, \eqref{term1pt2est}, and \eqref{term1pt3est}, we have
\begin{equation}\label{term5}
    \left| \int \lambda^{r-2}\partial_x\left[\lambda^2v\lambda^2(u+w)\right] \lambda^{r}vdx\right|
        \leq c_{r,s,\rho}\|v\|^2_{H^r}.
\end{equation}
To bound the second term of inequality \eqref{d2energy}, we apply the triangle inequality and Cauchy-Schwarz to get
\begin{equation}\label{F}
    \left| \int \lambda^{r-2}\partial_x\left[vz - v_xz_x\right]\lambda^{r}vdx\right| 
        \leq \|vz\|_{H^{r-1}}\|v\|_{H^{r}}
            + \|v_xz_x\|_{H^{r-1}}\|v\|_{H^{r}} .
\end{equation}
For $r - 1 \geq \frac12,$ Lemma \ref{HH} yields 
\begin{align}
    \left| \int \lambda^{r-2}\partial_x\left[vz - v_xz_x\right]\lambda^{r}vdx\right|
         &\leq c_r(\|u_0\|_{H^s} + \|w_0\|_{H^s})\|v\|^2_{H^r},
\end{align} 
when $r+1 \leq s$.
If rather $1 \leq r + 2 \leq 3, s - 3 \geq \frac12,$ and $r + 2 + s \geq 6$, then by Lemma \ref{lemma3} 
\begin{equation}
    \left| \int \lambda^{r-2}\partial_x\left[vz - v_xz_x\right]\lambda^{r}vdx\right|
        \leq c_r(\|z\|_{H^{s-3}}+ \|z_x\|_{H^{s-3}})\|v\|_{H^{r-1}}\|v\|_{H^{r}} ,
\end{equation}
and so the solution size estimates gives 
    \[\left| \int \lambda^{r-2}\partial_x\left[vz - v_xz_x\right]\lambda^{r}vdx\right|
        \leq 2c_{r,s}(\|u_0\|_{H^s} + \|w_0\|_{H^s})\|v\|^2_{H^r}.\]
Combining the above estimates we have for $1 \leq r \leq s-2, s > \frac72,$ and $r \geq 6-s$
\begin{equation}\label{1stterm}
    \left| \int \lambda^{r-2}\partial_x\left[\lambda^2v\lambda^2(u+w)\right] \lambda^{r}vdx\right| \leq c_{r,s,\rho}\|v\|^2_{H^r}.
\end{equation}
For the second term of inequality \eqref{energy}, we apply the triangle inequality and Cauchy-Schwarz to get
 
\begin{align}
 \notag   \left|\int \lambda^{r-4}\partial_x\left[\lambda^2v(\lambda^2(u+w) + \lambda^2v_x(\lambda^2(u+w)_x\right]\lambda^{r}vdx\right| 
        & \leq \left\|\lambda^2v(\lambda^2(u+w))\right\|_{H^{r-3}}\|v\|_{H^r} 
        \\ &+ \left\|\lambda^2v_x(\lambda^2(u+w)_x)\right\|_{H^{r-3}}\|v\|_{H^r}. \label{term2}
\end{align} 
If $\frac52 < r \leq s + 1,$ then by Lemma \ref{HH}, 
\begin{equation*} 
        \left\|\lambda^2v(\lambda^2(u+w))\right\|_{H^{r-3}} 
            \leq c_{r}\left\|\lambda^2(u+w)\right\|_{H^{r-2}}\|\lambda^2v\|_{H^{r-3}} 
                \leq c_r\|u+w\|_{H^s}\|v\|_{H^r},
\end{equation*}
and 
\begin{equation*} 
        \left\|\lambda^2v_x(\lambda^2(u+w)_x)\right\|_{H^{r-3}} 
            \leq c_{r}\left\|\lambda^2(u+w)_x\right\|_{H^{r-2}}\|\lambda^2v_x\|_{H^{r-3}} 
                \leq c_r\|u+w\|_{H^s}\|v\|_{H^r}.
\end{equation*}
Substituting the above estimates into \eqref{term2} and using the solution-size estimate, we have 
\begin{equation}\label{term2lem2est}
    \left|\int \lambda^{r-4}\partial_x\left[\lambda^2v(\lambda^2(u+w) + \lambda^2v_x(\lambda^2(u+w)_x\right]\lambda^{r}vdx\right| 
         \leq  c_{r,\rho}\|v\|_{H^r}^2
\end{equation}
If instead $1 \leq r \leq 3,$ Lemma \ref{lemma3} gives 
\begin{equation*} 
        \left\|\lambda^2v(\lambda^2(u+w))\right\|_{H^{r-3}} 
            \leq c_{r,s}\left\|\lambda^2(u+w)\right\|_{H^{s-3}}\|\lambda^2v\|_{H^{r-3}} 
                \leq c_{r,s}\|u+w\|_{H^s}\|v\|_{H^r},
\end{equation*}
and 
\begin{equation*} 
        \left\|\lambda^2v_x(\lambda^2(u+w)_x)\right\|_{H^{r-3}} 
            \leq c_{r,s}\left\|\lambda^2(u+w)_x\right\|_{H^{s-3}}\|\lambda^2v_x\|_{H^{r-3}} 
                \leq c_{r,s}\|u+w\|_{H^s}\|v\|_{H^r}.
\end{equation*}
Again, substituting into \eqref{term2} and using the solution-size estimate yields 
\begin{equation}\label{term2lem3est}
    \left|\int \lambda^{r-4}\partial_x\left[\lambda^2v(\lambda^2(u+w) + \lambda^2v_x(\lambda^2(u+w)_x\right]\lambda^{r}vdx\right| 
         \leq  c_{r,s,\rho}\|v\|_{H^r}^2.
\end{equation}
Combining estimates \eqref{1stterm}, \eqref{term2lem2est} and \eqref{term2lem3est}, we arrive at the following bound
\begin{equation}\label{finalenergy}
    \frac{1}{2}\frac{d}{dt}\|v\|^2_{H^r} \leq c_{r,s,\rho}\|v\|^2_{H^r}.
\end{equation}
Solving \eqref{finalenergy} yields 
\begin{equation}
    \|v(t)\|_{H^r} \leq e^{c_{r,s,\rho}T}\|v_0\|_{H^r},
\end{equation}
and so we have Lipschitz continuity in $A_1.$ \\

\textbf{H\"older continuity in $A_2.$}
Using the Lipschitz continuity in $A_1$ and the fact that $r \leq 6-s$ in $A_2$, we may write
\begin{equation}
    \|v(t)\|_{H^r} \leq \|v(t)\|_{H^{6-s}} \leq c_{r,s,\rho}\|v_0\|_{H^{6-s}}.
\end{equation}
Interpolating between the $H^r$ and $H^s$ norms then gives 
\begin{equation}
    \|v(t)\|_{H^r} \leq c_{r,s,\rho}\|v_0\|_{H^s}^{\frac{6-s-r}{s-r}}\|v_0\|_{H^r}^{\frac{2(s-3)}{s-r}} \leq c_{r,s,\rho}\|v_0\|_{H^r}^{\frac{2(s-3)}{s-r}}.
\end{equation}
Thus we have H\"older continuity in $A_2.$ \\ 

\textbf{H\"older continuity in $A_3.$} 
Since $s-1 \leq r \leq s$, we may interpolate between the $H^s$ and $H^{s-1}$ norms to get 
\begin{equation}
    \|v(t)\|_{H^r} \leq  \|v(t)\|_{H^{s-1}}^{s-r}\|v(t)\|_{H^s}^{r-s+1}.
\end{equation}
Since the solution size estimate gives 
\begin{equation}
    \|v(t)\|_{H^s} \lesssim \|u_0\|_{H^{s}} + \|w_0\|_{H^s} \lesssim \rho,
\end{equation}
we additionally have 
\begin{equation}
   \|v(t)\|_{H^r} \leq c_{\rho} \|v(t)\|_{H^{s-1}}^{s-r}.
\end{equation}
Therefore, by the Lipschitz continuity in $A_1$ and since $s-1 \leq r,$
\begin{equation}
    \|v(t)\|_{H^r} 
        \leq c_{\rho} \|v(t)\|_{H^{s-1}}^{s-r} 
            \leq c_{r,s,\rho} \|v_0\|_{H^r}^{s-r},
\end{equation}
and so we have the claimed  H\"older continuity in $A_3.$  

{\bf Acknowledgments.}
John Holmes thanks the NSF for support under grant DMS-2247019.

\end{document}